\documentclass[envcountsame,envcountsect]{svmult}

\usepackage{amsmath, amssymb, amsfonts}
\usepackage[all]{xy}

\usepackage{longtable}

\usepackage{mathptmx,helvet,courier,makeidx,graphicx,multicol,footmisc}

\newcommand{\Tr}{\operatorname{Tr}}
\newcommand{\Aut}{\operatorname{Aut}}
\newcommand{\Jac}{\operatorname{Jac}}
\newcommand{\End}{\operatorname{End}}
\newcommand{\PGU}{\operatorname{PGU}}
\newcommand{\SL}{\operatorname{SL}}
\newcommand{\FF}{\mathbb{F}}
\newcommand{\PP}{\mathbb{P}}
\newcommand{\ZZ}{\mathbb{Z}}
\newcommand{\QQ}{\mathbb{Q}}
\newcommand{\CC}{\mathbb{C}}

\begin{document}
\title{Zeta functions of a class of Artin--Schreier curves with many
  automorphisms} \author{Irene Bouw, Wei Ho, Beth Malmskog, Renate Scheidler,  Padmavathi Srinivasan, and Christelle Vincent}
\institute{Irene Bouw \at Institute of Pure Mathematics, Ulm University, D-89069 Ulm, \email{irene.bouw@uni-ulm.de}
  \and Wei Ho \at Department of Mathematics, University of Michigan, 530 Church Street, Ann Arbor, MI 48109, \email{weiho@umich.edu}
  \and Beth Malmskog \at Department of Mathematics and Statistics, Villanova University, 800 Lancaster Avenue, Villanova, PA 19085, \email{beth.malmskog@villanova.edu}
  \and
  Renate Scheidler \at    Department of Mathematics and Statistics, University of Calgary, 2500 University Drive NW, Calgary, Alberta, T2N 1N4, \email{rscheidl@ucalgary.ca}
  \and
Padmavathi Srinivasan \at Department of Mathematics, Massachusetts Institute of Technology, 77 Massachusetts Avenue, Cambridge, MA 02139, \email{padma\_sk@math.mit.edu}
\and Christelle Vincent \at 
Department of Mathematics, Stanford University, 450 Serra Mall, Building 380, Stanford, CA 94305, \email{cvincent@stanford.edu}
}
\authorrunning{Bouw, Ho, Malmskog, Scheidler, Srinivasan, and Vincent}
  \maketitle

\abstract{This paper describes a class of Artin--Schreier curves, generalizing results of Van der Geer and
Van der Vlugt to odd characteristic. The automorphism group of these curves contains a large
extraspecial group as a subgroup. Precise knowledge of this subgroup makes it possible to compute the
zeta function of the curves in this class over the field of definition of all automorphisms in the
subgroup.
%As a consequence, we obtain new examples of maximal curves. 
\keywords{2010 {\em Mathematics Subject Classification}. Primary 14G10. Secondary: 11G20,14H37.}}

\section{Introduction}\label{sec:intro}

In \cite{GeerVlugt}, Van der Geer and Van der Vlugt introduced a class
of Artin--Schreier curves over a finite field with a highly rich
structure. For example, these curves have a very large automorphism
group that contains a large extraspecial $p$-group as a
subgroup. Results of Lehr--Matignon \cite{LehrMatignon} show that the
automorphism groups of these curves are ``maximal'' in a precise
sense. (Lehr--Matignon call this a {\em big action}.)  A further
remarkable property is that all these curves are supersingular. This
yields an easy way of producing large families of supersingular
curves.

In \cite{GeerVlugt}, the authors explore these curves and their
Jacobians over fields of characteristic $2$. In this case, there is an
intriguing connection between the curves in this class and the weight
enumerator of Reed--M\"uller codes, which was their original
motivation for investigating this family of curves. In Sect. 13 of
\cite{GeerVlugt}, they sketched extensions of some of their results to
odd characteristic, but few details are given. The present paper
extends the main results and strategy of \cite{GeerVlugt} to the
corresponding class of curves in odd characteristic, providing full
details and proofs.

The main difference between the two cases is that
the aforementioned extraspecial group of automorphisms has exponent
$p$ in the case of odd characteristic $p$, whereas the exponent is $4$
in characteristic $2$. As a result, some of the arguments in the odd
characteristic case are more involved than those of \cite{GeerVlugt}.
%(such as Theorem 3.4 of \cite{GeerVlugt}, for example).
%\irene{I have deleted the ref to Theorem 3.4 of \cite{GeerVlugt},
%  since it plays no essential role in the argument that we are
%  generalizing, see Remark \ref{rem:isoclass}. I went through the
%  whole paper \cite{GeerVlugt}, and added remarks in the paper
%  pointing out the most prominent differences. }
Moreover, we
have streamlined the reasoning of \cite{GeerVlugt} and combined it
with ideas from \cite{LehrMatignon} to describe the automorphism group
of the curves under investigation.

%\irene{At a latter part of the paper, we should discuss the differences in more
%  detail. For example, Theorem 3.4 of \cite{GeerVlugt} does not hold
%  in odd characteristic. This is used in several places in the
%  original argument.}

%\renate{This is important and good to know. I've included this in the previous %paragraph; if there
%are more such examples, maybe we should list them. Moreover, whenever there is %a notable difference
%between a char $p$ result/proof of our paper and and char 2 result/proof given %in \cite{GeerVlugt},
%we should state that explicitly right before (or after) the result in question.% This strengthens
%our paper, and I worry about an unkind referee who thinks that we've simply cop%ied and done minor
%tweaks to the stuff in \cite{GeerVlugt}. (Maybe I sound paranoid, but I've had %some bad experience
%in that respect.)}
%\irene{The reason that I didn't write anything here so
%  far is that it is not so clear to me that there really are any
%  conceptual differences. One could say that the strategy is really
%  the same, and that there are no really new ideas if one is very
%  honest. We worked out some statements in more details, obtaining
%  cleaner statements. Also the formulas are more complicated for $p$
%  odd. This does not mean that writing the paper was straight forward
%  or easy, however. But one cannot write that. I would be grateful if
%  someone else has some suggestions.}

Arguably the most important object associated to an algebraic curve is
its zeta function since it encodes a large amount of information about
the curve, including point counts. Our main result is
Theorem \ref{thm:zeta} which computes the zeta function of the members of the family
of curves under consideration over a sufficiently large field. This
not only generalizes the corresponding result in \cite{GeerVlugt} for
characteristic $2$, but we also note that the authors of
\cite{GeerVlugt} do not offer an odd-characteristic analogue in their
paper. 

%Maximal curves have important uses in coding theory. 
The most prominent member of the family of curves considered in this
paper is the Hermite curve $H_p$ (Example \ref{exa:hermite}), which is
well known to be a maximal curve over fields of square cardinality. We
discuss other members of the family that are maximal in Sect.
\ref{sec:examples}. More examples along the same lines have also been
found by {\c{C}}ak{\c{c}}ak and {\"O}zbudak in
\cite{CO07}. 

We now describe the contents of this paper in more detail. Let $p$ be
an odd prime and $R(X)\in \overline{\FF}_p[X]$ be an additive polynomial
of degree $p^h$, i.e., for indeterminates $X$ and $Y$ we have
$R(X+Y)=R(X)+R(Y)$.  We denote by $C_R$ the smooth projective curve
given by the Artin--Schreier equation
\[
Y^p-Y=XR(X).
\]
The key to the structure of the curve $C_R$ is the bilinear form
$\Tr(XR(Y)+YR(X))$, introduced in Sect. \ref{sec:pointcount}, whose
kernel $W$ is characterized in Proposition
\ref{prop:B}, part \ref{item:prop:b-part2}. We obtain an expression for
the number of points of $C_R$ over a finite field in terms of
$W$. Over a sufficiently large field $\FF_q$ of square cardinality, we
conclude that the curve $C_R$ is either maximal or minimal,
i.e., either the upper or lower Hasse--Weil bound is attained (Theorem
\ref{thm:quadric} and part \ref{maxcharacterize} of Remark
\ref{rem:minmax}). To determine which of these
cases applies, we use the automorphisms of $C_R$.

In Sects. \ref{sec:S(f)} and \ref{sec:auto}, we show that $W$ also
determines a large $p$-subgroup $P$ of the group of automorphisms
(Theorem \ref{thm:aut}). With few exceptions, $P$ is the Sylow
$p$-subgroup of $\Aut(C_R)$ (Theorem \ref{thm:semigroups}).  It is an
extraspecial group of exponent $p$ and order $p^{2h+1}$, where
$\deg(R)=p^h$ (Theorem \ref{thm:extraspecial}).

In general, the size of the automorphism group restricts the
possibilities for the number of rational points of a curve.  In our
situation, there is a concrete relationship, since both the automorphisms
and the rational points of $C_R$ may  be described in terms of the
space $W$. We establish a point-counting result that applies to the
smallest field $\FF_q$ over which all automorphisms in $P$ are
defined.

The determination of the zeta function of $C_R$ over $\FF_q$ (Theorem
\ref{thm:zeta}) relies on a decomposition result for the Jacobian
$J(C_R)$ of $C_R$ (Proposition \ref{prop:KR}) that is an application
of a result of Kani--Rosen \cite{KaniRosen}. More precisely, we show
that $J(C_R)$ is isogenous over $\FF_q$ to the product of Jacobians of
quotients of $C_R$ by suitable subgroups of $P$ over $\FF_q$
(Proposition \ref{prop:KR}).  These quotient curves are twists of the
curve $C_{R_0}$ with $R_0(X)=X$ (Theorem \ref{thm:Aquotient}) for
which we may determine the zeta function by explicit point
counting. Putting everything together yields a precise expression for
the zeta function of $C_R$.

Our results also yield explicit examples of maximal curves (Sect.
\ref{sec:examples}). The main technical difficulty here is determining
the field $\FF_q$ over which all automorphisms in $P$ are defined.
%\irene{The remark on the fibre-product construction is now before
%  Exa.\ref{exa:hermite}.}
%\renate{I placed the sentence about the fiber product construction at the end o%f the examples, so
%we don't jus end abruptly on an example. Plus I am still concerned about advert%ising right at the
%beginning what we DIDN'T do. Stating this as future research also marks this as% territory.}

%\renate{I don't think these last two sentences should go here. We don't want to% say what we DIDN'T
%do. However, they could and probably should go into a yet to be written Conclus%ion/Open
%Problems/Future Research section at the end.}\irene{Do we want to write an open% problem etc section?
%I would not really know what to write.I suggest moving t%he last to lines as a remark to the example section.}

%\irene{Things we may want  to add: Non technical summary of the main result as
%  Theorem. Comparison with $p=2$, pointing out the differences in the
%  proofs.}

%\renate{Agreed; see also my comment further up.}
  % Thanks to Banff/organizers.}

\bigskip
\noindent {\em Acknowledgments.} This research began at the Women in
Numbers 3 workshop that took place April 20--25, 2014, at the Banff
International Research Station (BIRS) in Banff, Alberta (Canada). We
thank the organizers of this workshop as well as the hospitality of
BIRS. We also thank Mike Zieve for pointing out some references to us.

%\wei{add other thanks and grant stuff?}
IB is partially supported by
DFG priority program SPP 1489. WH is partially
supported by NSF grant DMS-1406066, and RS  is supported
by NSERC of Canada.
%\renate{If we get a kind referee, we should
%  eventually thank him or her.}

%\renate{Eventually, we will need a good deal more motivation than
%  this. Why are these curves interesting, etc? I liked Irene's
%  statement at WIN3, that many $\mathbb{F}_q$-automorphisms on a curve
%  always lead to many $\mathbb{F}_q$-rational points (via action and
%  orbits), but here, we have a connection the other way around.}

\subsection{Notation% and convention
}\label{sec:notation}

Let $p$ denote an odd prime, $\mathbb{F}_p$ be the finite field of order $p$, and $k =
\overline{\mathbb{F}}_p$ be the algebraic closure of $\mathbb{F}_p$.
%\irene{ I changed ``a fixed'' to ``the'', since the algebraic closure of $\FF_p$ is unique}.
All curves under consideration are assumed to be smooth, projective and absolutely irreducible.
Consider the %smooth projective
curve $C_R$ defined by the affine equation
\begin{equation}\label{eq:CR}
Y^p-Y = XR(X) ,
\end{equation}
where
\begin{equation*}
R(X) = \sum_{i=0}^{h} a_i X^{p^i} \in \mathbb{F}_{p^r}[X]
\end{equation*}
is a fixed additive polynomial of degree $p^h$ with $h \geq 0$ and whose coefficient field is denoted
$\mathbb{F}_{p^r}$. Note that $R$ is additive, i.e., $R(X+Y)=R(X)+R(Y)$ in $\mathbb{F}_{p^r}[X]$.
Thus, $C_R$ is defined over $\mathbb{F}_{p^r}$ and has genus
\begin{equation*}
g(C_R) = \frac{p^h(p-1)}{2} .
\end{equation*}
Of interest will be the polynomial $E(X)$ derived from $R(X)$ via
\begin{equation}\label{eq:E}
E(X) = (R(X))^{p^h} + \sum_{i=0}^{h} (a_i X)^{p^{h-i}} \in \FF_{p^r}[X]
\end{equation}
with zero locus
\begin{equation} \label{eq:W}
W = \{ c \in k : E(c) = 0 \} .
\end{equation}
Note that the formal derivative of $E(X)$ with respect to $X$ is the
constant non-zero polynomial $a_h$, so $E(X)$ is a separable additive
polynomial of degree $p^{2h}$ with coefficients in
$\mathbb{F}_{p^r}$. It follows that $W$ is an $\FF_p$-vector space of
dimension $2h$. When $h = 0$, i.e., $R(X) = a_0 X$, we have  $W = \{ 0 \}$. 
%(We point out that our situation
%differs from the setting of even characteristic in this respect, as in
%$E(X) = 0$ that case.)

We denote by $\mathbb{F}_q$ the splitting field of $E(X)$, so $W
\subset \mathbb{F}_q$. In Sect. \ref{sec:auto} of this paper we will
define and investigate a subgroup $P$ of the group of automorphisms of
$C_R$, and the automorphisms contained in $P$ will be defined over
this field $\mathbb{F}_q$.

\bigskip
For convenience, we summarize the most frequently used notation in Table \ref{tab:notation}.

\begin{center}
\begin{longtable}{|l|l|}
\caption{Frequently used notation} \label{tab:notation} \\
\hline
\textbf{Symbol} & \textbf{Meaning and place of definition} \\
\hline\hline
\endfirsthead
\multicolumn{2}{c}%
{\tablename\ \thetable\ -- \textit{Continued from previous page}} \\
\hline
\textbf{Symbol} & \textbf{Meaning and place of definition} \\
\hline
\endhead
\hline \multicolumn{2}{c}{\textit{Continued on next page}} \\
\endfoot
\hline
\endlastfoot
$p$ & an odd prime \\
$\mathbb{F}_{p^r}$ & field of definition of $R(X)$ and of $C_R$ (Sect. \ref{sec:notation}) \\
$\mathbb{F}_{p^s}$ & an arbitrary extension of $\mathbb{F}_{p^r}$ (Sect. \ref{sec:pointcount}) \\
$\mathbb{F}_q$ & $\mathbb{F}_q \supseteq \mathbb{F}_{p^r}$ splitting field of $E(X)$ (Sect. \ref{sec:notation}) \\
$k = \overline{\mathbb{F}}_p$ & algebraic closure of $\mathbb{F}_p$ (Sect. \ref{sec:notation})\\
$C_R$ & the curve $C_R : Y^p - Y = X R(X)$ over $\mathbb{F}_{p^r}$ (Eq. \ref{eq:CR})\\
$\overline{C}_A$ & quotient curve $C_R/A$ (Theorem \ref{thm:Aquotient}) \\
$R(X)$ & $R(X) = \sum_{i=0}^h a_i X^{p^i} \in \mathbb{F}_{p^r}[X]$ an additive polynomial (Eq. \ref{eq:CR})\\
$E(X)$ & $E(X) = (R(X))^{p^h} + \sum_{i=0}^h (a_iX)^{p^{h-i}} \in \mathbb{F}_{p^r}[X]$ (Eq. \ref{eq:E}) \\
$b, c$ & elements in $k$ with $b^p - b = cR(c)$ (Remark \ref{rem:b-identity}) \\
$B_c(X)=B(X)$ & polynomial s.t.\ $B(X)^p - B(X) = cR(X) + R(c)X$  \\
 & \hspace*{.85in} (Eqs. \ref{eq:B0} and \ref{eq:compareB}) \\
$W(\mathbb{F}_{p^s})$ & $W(\mathbb{F}_{p^s}) = \{ c \in \mathbb{F}_{p^s} :\Tr_{\mathbb{F}_{p^s}/\mathbb{F}_p}(cR(y)+y(R(c)) = 0 \text{ for all } y \in \mathbb{F}_{p^s} \}$ \\
& \hspace*{.85in} (Eq. \ref{eq:WFs}) \\
$W$ & $W = W(\mathbb{F}_q)$, space of zeros of $E(X)$ (Eq. \ref{eq:W}) \\

\parbox{0.5in}{$S(f)$ \\ } & \parbox{3.5in}{ $S(f)=\{(a,c, d)\in k^\ast \times k\times \FF_p^\ast : \text{there is } g\in k[X]$ s. t. \\
\hspace*{0.85in} $f(aX+c)-df(X)=g(X)^p-g(X) \}$ (Eq. \ref{eq:S(f)}) } \\
$\Aut^0(C_R)$ & group of automorphisms of $C_R$ that fix $\infty$
(Sect. \ref{sec:auto}) \\
$\sigma_{a,b,c,d}$ & automorphism in $\Aut^0(C_R)$ (Eq. \ref{eq:sigmaabcd}) \\
$\sigma_{b,c}$ & $\sigma_{b,c} = \sigma_{1,b,c,1}$ (Sect. \ref{sec:extraspecial}) \\
$\rho$ & Artin--Schreier automorphism, $\rho = \sigma_{1,1,0,1}$ (following Eq. \ref{eq:sigmaabcd}) \\
%the notation $\sigma_{c,b}$ for the same thing is no longer used. \\ &
%\padma{This comment about notation no longer being used is meant mostly for us, I suppose and will go away in the final paper?} \\
 $P$ & Sylow $p$-subgroup of $\Aut^0(C_R)$
(Theorem \ref{thm:aut}) \\ $H$ & $H = \Aut^0(C_R)/P$
(Theorem \ref{thm:aut}) \\ $Z(G)$ & center of a group $G$ \\ $E(p^3)$
& extraspecial group of order $p^3$ and exponent $p$
(Corollary \ref{cor:extraspecial}) \\ ${\mathcal A}$& a maximal abelian
subgroup of $P$ (Proposition \ref{prop:subgroups})\\ $J_R$ & $J_R
= \Jac(C_R)$, the Jacobian variety of $C_R$ \\
%$A_i$ & subgroups of ${\mathcal A}$ of order $p^h$ trivially intersecting each %other and $Z(P)$ \\
%\ & \hspace*{2.5in} (Proposition \ref{prop:subgroups}) \\
%$\overline{C}_{A_i}$& $\overline{C}_{A_i}=C_R/A$ is the quotient of $C_R$ by $A%_i$ above Lemma \ref{lem:KR}\\
%  & \hspace*{2.5in} (before Corollary \ref{cor:KR}) \\
$J\sim_\FF J'$& the ab. var. $J$ and $J'$ are isogenous over the field $\FF$ (Sect. \ref{sec:KR}).\\
$L_{C, \FF}(T)$ & numerator of the zeta function of the curve $C$ over the field $\FF$\\
& \hspace*{0.85in} (Sect. \ref{sec:zeta})
\end{longtable}
\end{center}

%\section{Point counting}\label{sec:pointcount}
\section{The kernel of the bilinear form associated to $C_R$}\label{sec:pointcount}

%\renate{I retitled this section since the point counting portion is really rather tangential. The bilinear form is much more important. I also think we need more than just $s \geq r$ here. We need $\mathbb{F}_{p^s}$ to contain $\mathbb{F}_{p^r}$ for these traces to be defined, so $s$ needs to be a multiple of $r$.}
%\xl{I think that $\mathbb{F}_{p^s}$ contains $\mathbb{F}_{p^r}$ when $s \geq r$, no?}\irene{No! The multiplicativity of the degree states  that $s=[\FF_{p^s}:\FF_p]=[\FF_{p^s}:\FF_{p^r}][\FF_{p^r}:\FF_p]$. Hence $s/r=[\FF_{p^s}:\FF_{p^r}]$ is an integer.}

Let $\mathbb{F}_{p^s}$ be any extension of $\mathbb{F}_{p^r}$. For each $s$ a multiple of $r$, we
associate to the curve $C_R$ the $s$-ary quadratic form
\begin{equation*}
x \mapsto \Tr_{\mathbb{F}_{p^s}/\mathbb{F}_p}(xR(x))
\end{equation*}
on $\mathbb{F}_{p^s}$, where $\Tr_{\mathbb{F}_{p^s}/\mathbb{F}_p} \colon
\mathbb{F}_{p^s} \rightarrow \mathbb{F}_p$ is the trace from the
$s$-dimensional vector space  $\mathbb{F}_{p^s}$
down to $\mathbb{F}_p$. The associated symmetric bilinear form on
$\mathbb{F}_{p^s} \times \mathbb{F}_{p^s}$ is
\begin{equation} \label{eq:bilinear}
(x,y) \mapsto \frac{1}{2}\Tr_{\mathbb{F}_{p^s}/\mathbb{F}_p}(xR(y)+yR(x)),
\end{equation}
%
%where we note that $2$ is invertible since $p$ is odd.
with kernel
\begin{equation} \label{eq:WFs}
W(\mathbb{F}_{p^s}) = \{ c \in \mathbb{F}_{p^s} : \Tr_{\mathbb{F}_{p^s}/\mathbb{F}_p}(cR(y)+yR(c)) = 0
    \text{ for all } y \in \mathbb{F}_{p^s} \}.
\end{equation}
Note that $W(\mathbb{F}_{p^s})$ is a vector space over $\FF_p$. The
following characterizations and properties of $W(\mathbb{F}_{p^s})$
will turn out to be useful.

\begin{proposition} \label{prop:B}
Let $c \in \mathbb{F}_{p^s}$. Then the following hold:
\begin{enumerate}
\item \label{item:prop:b-part1} If $c \in W(\mathbb{F}_{p^s})$, then
    $\Tr_{\mathbb{F}_{p^s}/\mathbb{F}_p}(cR(c)) = 0$.
\item \label{item:prop:b-part2} We have $c \in W(\mathbb{F}_{p^s})$ if and only if there exists
    a polynomial $B(X) \in \mathbb{F}_{p^s}[X]$ with
\begin{equation} \label{eq:B0}
B(X)^p-B(X)= cR(X)+R(c)X.
\end{equation}
Moreover, there is a unique solution $B_c(X) \in X\FF_{p^s}[X]$ to the equation (\ref{eq:B0}), and
\begin{enumerate}
\item \label{item:prop:b-part2a} The polynomial $B_c(X)$ is additive.
\item \label{item:prop:b-part2b} Every solution $B(X)$ of (\ref{eq:B0}) is of the form
    $B(X) = B_c(X) + \beta$ for some $\beta \in \FF_p$. 
%\padma{Calling a general solution also $B(X)$ is confusing, since two lines prior to this we say $B(X) = B_c(X) \in X\FF_{p^s}[X]$ -- we should call a general solution to this equation something else, or maybe it is best not to give it a name at all. }
\item \label{item:prop:b-part2c} If $c_1, c_2 \in W(\mathbb{F}_{p^s})$, then $B_{c_1+c_2}(X) =
    B_{c_1}(X)+B_{c_2}(X)$.
\end{enumerate}

%    \renate{Where do you need this ``additivity'' in $c$? Also, I am not sure I like it here
%      because it doesn't go with the ``$c \in \mathbb{F}_{p^s}$'' assumption at the beginning.
%      Can we make it a separate remark?}\irene{We use it in the proof of Prop.
%      \ref{prop:subgroups}. I really prefer leaving the statement as it is.} \xl{I do think
%      that this is the best place for this result, but I understand that it's a bit weird with
%      the beginning of the sentence... I put in subparts, does that make it better?} \renate{I
%      gave $c_1, c_2$ a place to live, and am happy with how this is now.}
\item \label{item:prop:b-part3} We have $c \in W(\mathbb{F}_{p^s})$ if and only if $E(c) = 0$,
    where $E(X)$ is the polynomial of (\ref{eq:E}) with zero locus $W$ as defined in
    (\ref{eq:W}). In other words, $W(\mathbb{F}_{p^s}) = W \cap \mathbb{F}_{p^s}$.
\end{enumerate}
\end{proposition}
\begin{proof} \hfill
\begin{enumerate}
\item Let $c \in W(\mathbb{F}_{p^s})$. Then substituting $y = c$ into (\ref{eq:WFs})
    yields $\Tr_{\mathbb{F}_{p^s}/\mathbb{F}_p}(2cR(c)) = 0$. Since $
    \Tr_{\mathbb{F}_{p^s}/\mathbb{F}_p}(X)$ is $\mathbb{F}_p$-linear and $p$ is odd, this
    forces $\Tr_{\mathbb{F}_{p^s}/\mathbb{F}_p}(cR(c))=0$.

\item The proof of part \ref{item:prop:b-part2} is analogous to that
  of Proposition 3.2 of \cite{GeerVlugt}. Assume that 
$c\in
  W(\FF_{p^s})$. We  show the existence of a solution $B$ of (\ref{eq:B0}), and show that statements \ref{item:prop:b-part2a}--\ref{item:prop:b-part2c} hold.  

We first recursively define numbers $b_i$
  using the following formulas:
%
%\irene{Sign error corrected in the next formula.}
\begin{align}
b_0 &= -c a_0 - R(c),\label{eq:B2} \\
b_i& = -ca_i + b_{i-1}^p \qquad \text{ for $1 \leq i \leq h-1$}\label{eq:B3},
\end{align}
and set $B_c(X) = \sum_{i=0}^{h-1} b_i X^{p^i}$. Then $B_c(X) \in X\FF_{p^s}[X]$, $B_c(X)$ is
additive, and $B_{c_1+c_2}(X) = B_{c_1}(X) +B_{c_2}(X)$ for all $c_1, c_2 \in
W(\mathbb{F}_{p^s})$. Furthermore, a simple calculation reveals that
\[ B_c^p(X)-B_c(X)=cR(X)+R(c)X+\epsilon X^{p^h}\]
with $\epsilon = b_{h-1}^p -ca_h \in \FF_{p^s}$. Note that
$\Tr_{\mathbb{F}_{p^s}/\mathbb{F}_p}(B_c(y)^p-B_c(y)) = 0$ for all $y \in \mathbb{F}_{p^s}$ by the additive version of Hilbert's Theorem $90$.

If $c \in W(\mathbb{F}_{p^s})$, then $\Tr_{\mathbb{F}_{p^s}/\mathbb{F}_p}(cR(y)+yR(c)) = 0$ for
all $y \in \mathbb{F}_{p^s}$, therefore $\Tr_{\mathbb{F}_{p^s}/\mathbb{F}_p}(\epsilon y^{p^h}) = 0$,
which forces $\epsilon = 0$. Hence $B_c(X)$ satisfies (\ref{eq:B0}), and
\begin{equation} \label{eq:B4}
b_{h-1}^p = ca_h .
\end{equation}
Moreover, if $B(X)$ is any solution to (\ref{eq:B0}), then $(B(X) - B_c(X))^p = B(X) - B_c(X)$,
so $B(X) - B_c(X) \in \FF_p$.

Conversely, if (\ref{eq:B0}) has a solution $B(X) \in \mathbb{F}_{p^s}[X]$, then
\[ 0 = \Tr_{\mathbb{F}_{p^s}/\mathbb{F}_p}(B(y)^p-B(y)) = \Tr_{\mathbb{F}_{p^s}/\mathbb{F}_p}(cR(y) +
R(c)y) \]
for all $y \in \mathbb{F}_{p^s}$, so $c \in W(\mathbb{F}_{p^s})$.

\item This result is stated for $p$ odd in Proposition 13.1 and
  proved for $p = 2$ in Proposition 3.1 of \cite{GeerVlugt}. It is
  also addressed in Remark 4.15 of the preprint \cite{LMpreprint} (the
  explicit statement is not included in \cite{LehrMatignon}, but can
  readily be deduced from the results therein). 
\end{enumerate}
 \hfill $\qed$
\end{proof}

\begin{remark}\label{rem:Bp=2}
The characteristic-$2$ analogue of Proposition \ref{prop:B} part \ref{item:prop:b-part2}  can be found in Sect. 3 of \cite{GeerVlugt}.  We also note that part \ref{item:prop:b-part1} of Proposition \ref{prop:B} does not hold in characteristic $p=2$ in general (see Sect. 5 of \cite{GeerVlugt}).
\end{remark}

Part \ref{item:prop:b-part3} of Proposition \ref{prop:B}
immediately establishes the following corollary.

\begin{corollary} \label{cor:W}
$W(\mathbb{F}_{p^s}) \subseteq W$, with equality for any extension $\mathbb{F}_{p^s}$ of the splitting field
$\mathbb{F}_q$ of $E$.
\end{corollary}

  We
conclude this section with a connection between the $\FF_p$-dimension
of the space $V_s = \mathbb{F}_{p^s}/W(\mathbb{F}_{p^s})$ and the
number of $\mathbb{F}_{p^s}$-rational points on the curve $C_R$. This
is obtained by projecting the bilinear form (\ref{eq:bilinear}) onto
$V_s$. We write $\overline{x} = x+ W(\mathbb{F}_{p^s})$ for the
elements in $V_s$.  Proposition \ref{prop:quadric} below is one of the key
ingredients in the determination of the zeta function of $C_R$ over
$\FF_q$ (Theorem \ref{thm:zeta}).

\begin{proposition} \label{prop:nondegenerate}
Define a map $Q_s$ on $V_s \times V_s$ via
\[ Q_s(\overline{x}, \overline{y}) = \frac{1}{2}\Tr_{\mathbb{F}_{p^s}/\mathbb{F}_p}(xR(y)+yR(x)) . \]
Then $Q_s$ is a non-degenerate bilinear form on $V_s \times V_s$.
\end{proposition}
\begin{proof}
%\padma{Do we really need to prove this in such great detail? It seems like a standard fact about how bilinear forms work. Going modulo the kernel of the bilinear form gives something non-degenerate. I hope I am not missing something subtle here.}
We begin by showing that $Q_s$ is well-defined. Let $x_1, x_2 \in \mathbb{F}_{p^s}$. Then
\begin{align*}
\overline{x}_1 = \overline{x}_2 & \Longleftrightarrow x_1 - x_2 \in W(\mathbb{F}_{p^s}) \\
& \Longleftrightarrow \Tr_{\mathbb{F}_{p^s}/\mathbb{F}_p}((x_1-x_2)R(y)+yR(x_1-x_2)) = 0 \mbox{ for all } y \in \FF_{p^s} \\
& \Longleftrightarrow \Tr_{\mathbb{F}_{p^s}/\mathbb{F}_p}(x_1R(y)+yR(x_1)) = \Tr_{\mathbb{F}_{p^s}/\mathbb{F}_p}(x_2R(y)+yR(x_2)) \mbox{ for all } y \in \FF_{p^s} \\
& \Longleftrightarrow Q_s(\overline{x}_1, \overline{y}) = Q_s(\overline{x}_2, \overline{y}) \mbox{ for all } \overline{y} \in V_s.
\end{align*}
Similarly, one obtains that $\overline{y}_1 = \overline{y}_2$ if and only if $ Q_s(\overline{x}, \overline{y}_1) = Q_s(\overline{x}, \overline{y}_2)
\mbox{ for all } \overline{x} \in V_s$. So if $(\overline{x}_1, \overline{y}_1) = (\overline{x}_2,
\overline{y}_2)$, then $Q_s(\overline{x}_1, \overline{y}_1) = Q_s(\overline{x}_1, \overline{y}_2) =
Q_s(\overline{x}_2, \overline{y}_2)$.

It is obvious that $Q_s$ is bilinear. To establish non-degeneracy, let
$\overline{x} \in V_s$ with $Q_s(\overline{x}, \overline{y}) = 0$ for
all $\overline{y} \in V_s$. Then
$\Tr_{\mathbb{F}_{p^s}/\mathbb{F}_p}(xR(y)+yR(x)) = 0$ for all $y \in
\FF_{p^s}$, so $x \in W(\FF_{p^s})$, and hence $\overline{x} =
\overline{0}$.
\hfill $\qed$ \end{proof}

It follows that the quadratic form  $\overline{x} \mapsto Q_s(\overline{x}, \overline{x})$ on $V_s$
is non-degenerate. Therefore, its zero locus
\begin{equation*}
\{ \overline{x} \in V_s : \Tr_{\mathbb{F}_{p^s}/\mathbb{F}_p}(xR(x))  = 0\}
\end{equation*}
defines a smooth quadric over $\mathbb{F}_{p}$.

In \cite{Joly}, Joly provides a formula for the cardinality of the
zero locus of a non-degenerate quadratic form, which we reproduce here
for the convenience of the reader. The case of $n$ odd is treated in
Chap. 6, Sect. 3, Proposition 1, and the case of $n$ even  is Proposition 2 of Chap. 6, Sect. 3.
%\renate{I think this is the wrong use of
%``loc cit.'' because that is meant to refer to the same page number as before. %You want
%\emph{ibid.}} \xl{I decided that in the end it was clear enough that I meant in% his paper so I just
%took it out.}
%We further note that the results is proved in \cite{Joly} for an arbitrary finite
%field in place of $\mathbb{F}_p$, replacing $p$ below by the cardinality of this finite field, but
%we state only the result we will need here.
Note that in \cite{Joly}, the result is proved for forms over an arbitrary finite field, but we
restrict to $\mathbb{F}_p$ here which is sufficient for our purpose.

%\renate{I find it very confusing that this is stated over $\mathbb{F}_{p^s}$, when we actually apply it over $\mathbb{F}_{p^s}/W(\mathbb{F}_{p^s})$. In fact, I think we just need this result over $\FF_p$, so that's how I reproduced it. (The original version s just commented out.)}
%\irene{We may want to remark that every quadratic form over $\FF_p$ is diagonalizable. Here we use that $p$ is odd. This is \cite{Cassels}, Theorem 3.1 of Chapter 8. This is not true for $p=2$.}
%\xl{I have incorporated this last fact in my comment above, which I think that ultimately (in a couple of days when we've all looked at this and agreed that we're not going crazy) should go into the proof of Proposition 2.4)}
\begin{theorem}[Joly \cite{Joly}] \label{thm:quadric}
Let $a_1 X_1^2 + \cdots + a_n X_n^2$ be a non-degenerate quadric in $n$ variables with coefficients
in $\FF_p$, and $N$ be the cardinality of its zero locus. Then
\begin{equation*}
N = \begin{cases}
p^{n-1} & \textup{if $n$ is odd}, \\
p^{n-1} + (p^{n/2} -p^{n/2-1}) & \textup{if $n$ is even and $(-1)^{n/2}a_1 \cdots a_n \in (\mathbb{F}_p^{\ast})^2$}, \\
p^{n-1} - (p^{n/2} -p^{n/2-1}) & \textup{if $n$ is even and $(-1)^{n/2}a_1 \cdots a_n \notin (\mathbb{F}_p^{\ast})^2$}.
\end{cases}
\end{equation*}
\end{theorem}

Applying this result to the quadric $x \mapsto
\Tr_{\mathbb{F}_{p^s}/\mathbb{F}_p}(xR(x))$ on the space
$\mathbb{F}_{p^s}/W(\mathbb{F}_{p^s})$, we obtain the following point
count for the curve $C_R$. This result is already presented in \cite{GeerVlugt}, but we include it here to provide a proof.
%For a concrete example we refer to Example
%\ref{exa:quadric}.

\begin{proposition}[Proposition 13.4 of \cite{GeerVlugt}]\label{prop:quadric}
Let $w_s = \dim_{\FF_p}(W(\FF_{p^s}))$ and $n_s = s-w_s$. Then the number of
$\mathbb{F}_{p^s}$-rational points on $C_R$ is
\begin{equation*}
\#C_R(\mathbb{F}_{p^s}) =
\begin{cases}
p^s+1 & \quad \text{for $n_s$ odd},\\
{p^s}+1 \pm (p-1)\sqrt{p^{s+w_s}} & \quad \text{for $n_s$ even},
\end{cases}
\end{equation*}
with the sign depending on the coefficients of the quadratic form $Q_s$.
\end{proposition}
\begin{proof}
We have $V_s = \mathbb{F}_{p^s}/W(\mathbb{F}_{p^s}) \simeq \mathbb{F}_p^{n_s}$, where $n_s = s -
w_s$. Therefore, for $\bar{x} \in V_s$, we may write $\bar{x} = (x_1, \ldots, x_{n_s})$, with each
$x_i \in \mathbb{F}_p$. %Adopting this point of view,
In this way, $Q_s(\overline{x}, \overline{x})$ on the space $V_s$ is a non-degenerate quadric in
$n_s$ variables with coefficients in $\FF_p$. Furthermore, it is diagonalizable by \cite[Chap. 8, Theorem 3.1]{Cassels} since $p$ is odd, and therefore can be written in the form $\sum_{i =
1}^{n_s} a_i X_i^2$ with $a_i \in \mathbb{F}_p$ for $1 \leq i \leq n_s$. As a consequence we may
apply Theorem \ref{thm:quadric} to obtain the cardinality of the set
\begin{equation*}
\{ \overline{x} \in V_s \simeq \mathbb{F}_p^{n_s}: Q_s(\overline{x},\overline{x}) = 0 \} = \{ \overline{x} \in V_s :  \Tr_{\mathbb{F}_{p^s}/\mathbb{F}_p}(xR(x))  = 0\}.
\end{equation*}

Each $\overline{x} \in V_s$ with $Q_s(\overline{x}, \overline{x}) = 0$ gives rise to $p^{w_s}$
distinct values $x \in \mathbb{F}_{p^s}$ such that $\Tr_{\mathbb{F}_{p^s}/\mathbb{F}_p}(xR(x)) =
0$. For each of these $x \in \mathbb{F}_{p^s}$, we have $p$ solutions $y$ to the equation $y^p-y =
xR(x)$. In addition to these points, $C_R$ has one point at infinity which is defined over any
extension of $\mathbb{F}_{p^r}$. Hence $\#C_R(\mathbb{F}_{p^s}) = p^{w_s+1}N + 1 $ with $N$ given
as in Theorem \ref{thm:quadric} (with $n = n_s)$.
\hfill $\qed$ \end{proof}

Note that a more general version of Proposition \ref{prop:quadric} can be
found in Theorem 4.1 of \cite{CO08}.

%\irene{I suggest writing as remark or cor the statement for $q$ (or extensions %of $\FF_q$). vdGvdV
%claim that $E$ is irreducible, I  believe, this would yield $q=p^{2hr}$, which %is in particular
%even.}

%\xl{What statement should we write for $q$? Proposition 2.4?}

%\xl{About $E$, the statement in GeerVlugt is that $E$ is irreducible over the p%urely transcendental
%function field $\overline{\mathbb{F}}_q(a_0,a_1, \ldots, a_h)$, where the $a_i$%'s are the
%coefficients of $R$. If there is interest I can type up the proof of that state%ment.} \irene{OK, I
%misremembered, that does not help really for the zeta function.} \renate{In fac%t, $E(X)$ cannot be
%irreducible as it is a multiple of $X$. I don't see at this point how Prop. 13.%2 of vdGvdV could
%help.}

\section{Connection to automorphisms of $C_R$}\label{sec:S(f)}

In this section, we generalize the results of Proposition \ref{prop:B}
to lay the groundwork for our investigation of the $k$-automorphisms
of $C_R$ that stabilize $\infty$, the unique point at infinity on
$C_R$. We follow Sect. 3 of \cite{LehrMatignon}, but our notation is slightly
different. Similar results may also be found in \cite{Elkies}.

We define for any polynomial $f(X)
\in k[X]$ the set
\begin{align}\label{eq:S(f)}
S(f(X))=\{(a,c, d)\in k^\ast & \times k\times \FF_p^\ast :
\text{there exists $g(X)\in X k[X]$ such that} \nonumber \\ &
f(aX+c)-df(X) = g(X)^p - g(X) \}.
\end{align}
In our situation we take $f(X)=XR(X)$, where $R(X)$ is an additive
polynomial of degree $p^h$.
%\irene{shouldn't we get rid of $f$ alltogether?} \xl{I don't think it hurts to %make the general definition first.}
It is easy to verify that if $(a,c,d) \in
S(XR(X))$ then the map $(x,y) \mapsto (ax+c, dy + g(x))$ is an
automorphism of $C_R$ that fixes $\infty$. In fact, in Lemma
\ref{lem:aut} we will see that every automorphism of $C_R$ that fixes
$\infty$ is of this form. The elements in $S(X R(X))$, along with
the polynomial $g(X)$, can be characterized explicitly as follows.

\begin{proposition} \label{prop:h=0}
If $h = 0$, then $S(X R(X)) = \{ (a, 0, a^2) : a^2 \in \mathbb{F}_p^\ast \}$.
\end{proposition}
\begin{proof}
If $h = 0$, then $R(X) = a_0X$, so
\[ (aX+c)R(aX+c) - dXR(X) = a_0\left((a^2-d)X^2 + 2acX + c^2\right). \]
This polynomial is of the form $g(X)^p - g(X)$ if and only if $g(X)^p - g(X) = 0$, or equivalently,
$a^2 = d$, $c = 0$ and $g(X) \in \FF_p$.
\hfill $\qed$\end{proof}

%\irene{The numbering in the next prop is inconsistent with what we did
%  elsewhere, but I did not have the courage to change the references
%  in the whole paper.}
%\xl{I tried to address this as much as possible by using "Find" to find when we% have referred to prop:uniqueness. I used the "label" command to label the parts
%of the proposition which we use later, hopefully this will help if we need to% make more changes in the future.}

\begin{proposition}\label{prop:uniqueness} \hfill
\begin{enumerate}
\item \label{item:compare} Assume that $h \geq 1$ and let $a \in k^\ast$, $c \in k$ and $d \in
    \FF_p^\ast$. Then $(a,c,d) \in S(X R(X))$ if and only if there exists $B(X) \in Xk[X]$ such
    that
\begin{equation}\label{eq:compareB}
cR(X) + R(c)X  = B(X)^p - B(X) ,
\end{equation}
and %$a$ and $d$ satisfy
\begin{equation}\label{eq:compareR}
aR(aX) = dR(X).
\end{equation}
%cR(c) &= b^p - b,\label{eq:compareb}\\
%The last equation is equivalent to $d = a^{p^i+1}$ for all $0 \leq i
%\leq h$ with $a_i \neq 0$.
%\irene{I included (\ref{eq:compareR}) into
%  the display.}
%\xl{Actually logically it's pretty silly to include the condition on $b$ in the% if and only if statement since we can always find such a $b$. It's not like
%this is a requirement, so I took it out.}
%
\item \label{item:prep:uniqueness-part1} If the equivalent conditions of part \ref{item:compare}  are fulfilled,
    then $c$ and $B(X)$ satisfy the following conditions.
\begin{enumerate}
\item \label{cinW} $c \in W$.
\item \label{additive} The polynomial $B(X)=B_c(X)$ only depends on $c$ and is uniquely
    determined by (\ref{eq:compareB}) and the condition that $B_c(X) \in Xk[X]$. It is an
    additive polynomial with coefficients in $\FF_{p^r}(c) \subseteq \FF_q$.
%The polynomial
 % $B_c(X)$ is also additive considered as polynomial in $c$.
 % \xl{Actually for $B$ to be unique it needs to not have a constant term. Other%wise there could be a constant term in $\mathbb{F}_p$.}
\item \label{degreeB} The polynomial $B_c(X)$ is identically zero if and only if $c=0$, and
    has degree $p^{h-1}$ otherwise.
\end{enumerate}
\item \label{convenientb} For a triple $(a,c,d) \in S(XR(X))$, all polynomials $g(X)$ as given
    in (\ref{eq:S(f)}) are of the form
\begin{equation*}
g(X) = B_c(aX) + \frac{B_c(c)}{2} + i,
\end{equation*}
as $i$ ranges over $\FF_p$. In particular, each of these polynomials $g(X)$ has coefficients in
$\FF_q(a)$.
\end{enumerate}
\end{proposition}
%\irene{I don't understand the statement here. How can the uniqueness
%  of $B$ be a condition on $a,b,c,d$? Or do you want to say that
%  $(a,c,d\in S(f)$ if and only if there exists $B, b$ satisfying
%  (\ref{eq:compare}). If $a,b,d\in S(f)$, then the following statements
%  hold for the corresponding $B, b$ etc. However since $B, b$ only
%  depend on $c$ I find it hard to see how this can determine the
%  possibilities for $a, d$.  }
%\xl{I rephrased the latter part of this proposition as more of a laundry list o
%f properties satisfied by $a,b,c,d,B$, and reworked the proof which I hope clarifies some of this.}
 % \irene{To do: We have to check whether some of this can (implicitly)
 %   be found in \cite{LehrMatignon}.} \padma{Since this is only going to appear in a conference proceedings, if these statements are not stated as explicitly in \cite{LehrMatignon} then maybe it is okay to leave this on since it is not standard material?}

%\irene{Renate distinguished between $B$ for an arbitrary sol and $B_c$ for the %unique additive solution. The current formulation no longer respects this conve%ntion.}
%\xl{I put in the $_c$ in this section because later on it is there and it is le%ss work to put in the $_c$ here than take it out everywhere else. Besides I do
%think that we felt it was important that $B$ depended on $c$.}

\begin{proof} \hfill
\begin{enumerate}
 \item 
Let $(a,c,d)\in k^\ast\times k \times \FF_p^\ast$. % be an arbitrary tuple.
Suppose first that there exists $B(X) \in Xk[X]$ satisfying  (\ref{eq:compareB}), and that
$a$ and $d$ satisfy (\ref{eq:compareR}). % for our given polynomial $R(X)$.
%\xl{I think the
%  condition \ref{eq:compareR} is a condition on $a$ and $d$, not $R$,
%  no?}
Then for any $b \in k$ such that $b^p - b = cR(c)$, we have
\begin{align*}
(aX+c)R(aX+c) - dXR(X) & = X(aR(aX)- dR(X)) +cR(aX) + aXR(c)  +cR(c)  \\
& = B(aX)^p-B(aX) + b^p -b,
\end{align*}
and so we may take $g(X) = B(aX) +b$ to show that $(a,c,d) \in S(XR(X))$.
%\beth{Changed this so I believe it is correct now... check me!  Changed because the condition for being in $S(f(X))$ is that $f(aX+c)-df(x)=g(X)^p-g(X)$.  Since $f(X)=XR(X)$, we have $f(aX+c)-df(X)=(aX+c)R(aX+c)-dXR(X)$.  Now everything works with that.  } \xl{I think that there was still a small typo ($d$ instead of $a$ I think) but now I think that its okay.}

Conversely, suppose that $(a,c,d) \in S(XR(X))$. Then there exists a polynomial $g(X) \in k[X]$ such that
\begin{equation*}
X(aR(aX)-dR(X)) + cR(aX) + aR(c)X + cR(c) = g(X)^p - g(X).
\end{equation*}
Writing $g(X) = b + \widetilde{B}(X)$ with $\widetilde{B}(X) \in Xk[X]$, we see that this is
equivalent to the existence of a polynomial $\widetilde{B}(X) \in X k[X]$ such that
\begin{equation} \label{eq:BFG}
\widetilde{B}(X)^p - \widetilde{B}(X) = X F(X) + G(X)
\end{equation}
where $F(X) = aR(aX) - dR(X)$ and $G(X) = cR(aX) + aR(c)X$ are both additive polynomials. We note
for future reference during the proof of part \ref{convenientb} that this also implies $b^p - b = cR(c)$.

%\renate{I don't like referring to results hidden in proofs. I think
%  the identity $b^p - b = cR(c)$ should be stated as one of the
%  results in this proposition. As far as I can tell, we use this quite
%  a bit throughout the paper. Can we include this in part \ref{cinW},
%  e.g..\ ``$c \in W$ and there exists $b \in K$ with $b^p - b =
%  cR(c)$. In fact we use this so frequently that it should get its own
%  equation number.}  \xl{I have made a remark outside of the proof,
%  with an equation that is numbered. I don't like saying that ``there
%  exists $b \in K$ with $b^p - b = cR(c)$" in the statement of the
%  theorem because of course there exists such a $b$ if you are willing
%  to take a field extension. This is why I took it out of the
%  statement: we don't need to hypothesize or prove the existence of
%  such $b$'s.}  \irene{I agree with Christelle that the remark is
%  clear. Alternatively, maybe we could say in part (3) that the
%  $b=B(c)/2+i$ from that statement satisfies $b^p-b=cR(c)$. This is
%  really the thing we need.} \renate{I like how it is now, thanks, Christelle.
%  I did a quick scan of the paper to see where we use this and referred
%  to this remark by number a couple of times, just to give it some context.
%  I also updated the notation table to refer to the remark.}

Note that (\ref{eq:compareR}) holds if and only if $F(X) = 0$, in which case $B(X) =
\widetilde{B}(X/a) \in Xk[X]$ satisfies (\ref{eq:compareB}). Thus, it suffices to show that
$(a,c,d) \in S(X R(X))$ implies $F(X) = 0$ to complete the proof of part \ref{item:compare}.

To this end, we note that all the monomials in $XF(X)$ and $G(X)$ are
of the form $X^{p^i+1}$ and $X^{p^i}$ for $0 \leq i \leq h$. If
$\widetilde{B}(X) = 0$, then this immediately forces $F(X) = G(X) =
0$, so assume that $\widetilde{B}(X) \neq 0$.

Comparing degrees in (\ref{eq:BFG}) shows that $\deg(\widetilde{B})
\leq p^{h-1}$. Put
\begin{equation*}
\widetilde{B}(X) = \sum_{j=1}^{p^{h-1}} \tilde{b}_j X^j, \quad
\tilde{b}_j \in k \mbox{ for } 1 \leq j \leq p^{h-1} ,
\end{equation*}
and consider the polynomial $\widetilde{B}(X)^p - \widetilde{B}(X)$. In this polynomial, the coefficient of $X^j$  for
$1 \leq j \leq p^h$ is
\[ \begin{cases}
-\tilde{b}_j & \mbox{when $p\nmid j$},
\\ \tilde{b}_{j/p}^p-\tilde{b}_j & \mbox{when $p\mid j$
  and $j \leq p^{h-1}$}, \\ \tilde{b}_{p^{h-1}}^p & \mbox{when $j =
  p^{h}$}.
\end{cases} \]
All coefficients of $X^j$ for $j\neq p^i, p^i+1$ must vanish.
% so the
%only potentially non-zero coefficients of $\widetilde{B}(X)^p -
%\widetilde{B}(X)$ are $-\tilde{b}_1$, $-\tilde{b}_{p^i+1}$ for $0 \leq
%i \leq h-2$, $\tilde{b}_{p^{i-1}}^p -\tilde{b}_{p^i}$ for $1 \leq i
%\leq h-1$, and $\tilde{b}_{p^{h-1}}$.
We conclude that the coefficients $\tilde{b}_j$ of $\widetilde{B}(X)$ are zero for all $j\neq
p^i,p^i+1$, so we may write $\widetilde{B}(X) = XU(X) + V(X)$ where $U(X), V(X) \in k[X]$ are
additive polynomials. Then (\ref{eq:BFG}) yields
\[ X^p U(X)^p + V(X)^p - X U(X) - V(X) = X F(X) + G(X) . \]
Except for the monomials in $X^p U(X)^p$, this polynomial identity
only contains monomials of the form $X^{p^i}$ and $X^{p^i+1}$; the
monomials in $X^p U(X)^p$ all take the form $X^{p+p^{i+1}}$.  This
forces $U(X) = 0$. Thus, $X F(X) = V(X)^p - V(X) - G(X)$ is additive,
which is only possible if $F(X) = 0$.  %The proof of part \ref{item:compare} is nowcomplete.

\item
The proof of part \ref{item:prep:uniqueness-part1} is now straightforward.  We remark that equation (\ref{eq:compareB}) is identical
to equation (\ref{eq:B0}). Therefore \ref{cinW} follows from part \ref{item:prop:b-part2} of Proposition \ref{prop:B}, and $B(X)$ is
identical to the polynomial $B_c(X)$ defined in that proposition since $B(X) \in Xk[X]$. Thus,
$B(X)$ only depends on $c$ and is unique, and we write $B_c(X)$ for this polynomial from now on.
The additivity of $B_c(X)$ was already  established in the proof of part \ref{item:compare}, since
$B_c(X)=\widetilde{B}(X/a)$, and $\tilde{B}(X)=V(X)$ is additive; note that it also follows from part \ref{item:prop:b-part2a} of 
Proposition \ref{prop:B}.  Moreover, the coefficients of $B_c$ satisfy
(\ref{eq:B2})--(\ref{eq:B4}) and thus belong to $\FF_{p^r}(c)$. Part \ref{item:compare} and Corollary
\ref{cor:W} imply that $\FF_{p^r}(c) \subseteq \FF_q$. This proves \ref{additive}.

If $c = 0$, then $B_c(X)=0$. If $c\neq 0$, the polynomial $B_c(X)$ is
obviously nonzero and (\ref{eq:B4}) shows that $B_c(X)$ has degree
$p^{h-1}$. This proves \ref{degreeB}.

\item Writing $g(X) = b + \widetilde{B}(X)$ with $\widetilde{B}(X) \in Xk[X]$ as in
the proof of part \ref{item:compare}, we have already seen that $B_c(X) = \widetilde{B}(X/a)$, and $b$ is any
solution to the equation $b^p-b = cR(c)$. Any two such solutions differ by addition of an element
in $\mathbb{F}_p$. Furthermore, since $2\in \FF_p^\ast$, it follows from (\ref{eq:compareB}) that
$b=B_c(c)/2$ satisfies $b^p-b=cR(c)$, and the first statement of part \ref{convenientb} follows. The second statement
of part \ref{convenientb} follows from part \ref{additive}.
\end{enumerate}
\hfill $\qed$\end{proof}

\begin{remark} \label{rem:b-identity}
We repeat here a remark made in the proof since we will use this throughout the paper. For a triple
$(a,c,d) \in S(XR(X))$, all polynomials $g(X)$ as given in (\ref{eq:S(f)}) can be written as
\begin{equation*}
g(X) = B_c(aX) + b,
\end{equation*}
where $B_c(aX) \in \FF_q(a)$, and $b \in k$ is a solution of the equation
\begin{equation}
b^p - b = cR(c).
\end{equation}
Part \ref{convenientb} of Proposition \ref{prop:uniqueness} implies that
every solution $b$ of this equation is of the form $b=B_c(c)/2+i$ with
$i\in \FF_p$.
\end{remark}

\section{Automorphism group of $C_R$}\label{sec:auto}

%\renate{In light of Proposition \ref{prop:uniqueness} and Corollary \ref{cor:1c%1-a0d}, some of this
%can likely now be shortened.}
%\irene{I changed the two subsections to two different sections, as this
%  is more consistent with the rest of the paper.}

%\subsection{Automorphisms of $C_R$}\label{subsec:autoI}

In this section we apply the results of the previous section to study
the group $\Aut(C_R)$ of $k$-automorphisms of the curve $C_R$, and
more particularly the subgroup $\Aut^0(C_R)$ of automorphisms of $C_R$
that fix the unique point at infinity, i.e., the unique point of $C_R$
which does not belong to the affine curve defined by (\ref{eq:CR}).
%\padma{Minor point: Is it clear that we mean the preimage of $\infty$ in $\mathbb{P}^1$ when we think of this curve as a cover of $\mathbb{P}^1$ via the map given by the equation and the variable $X$ or should we slightly elaborate on this a little? There are many different maps from the curve $C_R$ to $\mathbb{P}^1$ and we mean the point lying over $\infty$ for a very specific map to $\mathbb{P}^1$}, which we denote by $\infty$.
The main result is Theorem \ref{thm:aut}, which describes
$\Aut^0(C_R)$.

Recall from Sect. \ref{sec:S(f)} that to a triple $(a,c,d)\in S(XR(X))$ we
associate the $k$-auto\-morphism
\begin{equation}\label{eq:sigmaabcd}
\begin{aligned}
\sigma_{a,b,c,d}    \colon C_R & \rightarrow C_R \\
  (x,y) & \mapsto (ax+c, dy+b+B_c(ax))
\end{aligned}
\end{equation}
of $C_R$. Here $b$ is a solution of the equation $b^p-b = cR(c)$ (see Remark \ref{rem:b-identity})
and $B_c$ is as in Proposition \ref{prop:uniqueness}.  Note that $\sigma_{a,b,c,d}$ fixes the point
$\infty$. In the rest of the paper, we denote by
\begin{equation*}
\rho(x,y)=\sigma_{1,1,0,1}(x,y)=(x,y+1)
\end{equation*}
the Artin--Schreier automorphism of the curve $C_R$.

The following lemma summarizes some properties of the automorphisms
$\sigma_{a,b,c,d}$.

\begin{lemma}\label{lem:aut}
With the above notation and assumptions, we have
\begin{enumerate}
\item \label{item:stab} Every element of the stabilizer $\Aut^0(C_R)$ of the point $\infty$ is
    of the form $\sigma_{a,b,c,d}$ as in (\ref{eq:sigmaabcd}).
\item \label{item:order} The automorphisms $\sigma_{1,b,c,1}$ with $(b,c)\neq (0,0)$ have order
    $p$. For $(a,d)\neq (1,1)$ the order of $\sigma_{a,b,c,d}$ is not a $p$-power.
\end{enumerate}
\end{lemma}

\begin{proof}  The lemma follows from Corollaries 3.4 and 3.5 in
\cite{LehrMatignon}. We recall the proof.

\begin{enumerate}
 \item 
Part \ref{item:stab} follows from Proposition 3.3 of \cite{LehrMatignon} in the case that
$g(C_R)\geq 2$. (Since $p$ is odd in our set-up and the genus of $C_R$ is $p^h(p-1)/2$, this only excludes the case that $h=0$ and $p=3$. This case is
treated in the proof of Corollary 3.4 of \cite{LehrMatignon}.) Namely, let $\varphi\in \Aut^0(C_R)$
be an automorphism of $C_R$ fixing $\infty$. Then the proof of Proposition 3.3 of
\cite{LehrMatignon} shows that there exists an isomorphism $\tilde{\varphi} \colon \PP^1\to \PP^1$
together with a commutative diagram
\[
\xymatrix{  C_R\ar[r]^{\varphi}\ar[d]&C_R\ar[d]\\
       \PP^1\ar[r]^{\tilde{\varphi}}&\PP^1,
}
\]
where the vertical maps are $(x,y)\mapsto x$.

The morphism $\tilde{\varphi}$ fixes $\infty\in \PP^1$, hence it is an affine
linear transformation and  we may write it as
$\tilde{\varphi}(x)=ax+c$ with $a\in k^\ast$ and $c\in k$. The
commutative diagram above implies that $\varphi(x,y)=(ax+c,dy+
g(x))$ for some $g(X)\in k(X)$ and $d\in k^\ast$.
 The assumption that $\varphi$ fixes the point $\infty$ implies that
 $g(X)\in k[X]$ is a polynomial. The statement that
 $\varphi=\sigma_{a,b,c,d}$ follows since $\varphi$ is assumed to be
 an automorphism of $C_R$.  %This proves part \ref{item:stab}.

 \item
To prove part \ref{item:order} we first remark that if $\sigma_{a,b,c,d}$ has $p$-power order, then
$a=d=1$, since $1$ is the only $p$th root of unity in $k$.  We  show that every nontrivial
automorphism $\sigma_{1,b,c,1}$ has order $p$.

We compute that
\[
\sigma_{1,b,c,1}^p(x,y)=(x+pc,y+pb+B_c(x)+B_c(x+c)+\cdots+B_c(x+(p-1)c)).
\]
Recall from Proposition \ref{prop:uniqueness} that $B_c$ is an additive
polynomial; in particular, its  constant term vanishes. Hence
\[
B_c(X)+B_c(X+c)+\cdots+B_c(X+(p-1)c)=\sum_{i=0}^{p-1}
B_c(ic)=\sum_{i=0}^{p-1}i B_c(c)=0.
\]
This implies that $\sigma_{1,b,c,1}^p=1$.
\end{enumerate}
\hfill $\qed$\end{proof}

 % \irene{Let me know if this clarifies things. I believe the problem
 %   is that the LehrMatignon paper is confusingly formulated and
 %   contains typos.  If it is still unclear the reference to Satz 4 in
 %   [ST2] may help. To see in what fields in coefficients live, one
 %   just has to work out the condition that $\tilde{X}, \tilde{Y} $
 %   satisfy the equation. For example $d\in \FF_p^\ast$ comes from
 %   $d^p=d$ which one needs to cancel this term. The rest of this
 %   section needs to be slightly restructured, since I reformulated
 %   the statements.}

%\beth{I think that the automorphism $\sigma_{1,b,0,1}$ has order $p$
%  as well, where $b\in\mathbb{F}_p^{\ast}$.  Aren't these just the
%  Artin-Schreier automorphisms?  This way of thinking of it makes it
%  clearer to me!}

%\xl{$c$ is not allowed to be 0, see definition of $S(f)$ in equation (2.1).}

%\xl{Our part (b) doesn't follow from LehrMatignon: They allow $c =0$,
%  and they put in $\rho$, the Artin-Schreier automorphism. Our (b)
%  might still be true but there is an argument to be made here.}

%\xl{Renate and I think we need to just give the proof here, even if we
%  end up cutting it from the final version of the paper. This would
%  unconfuse all of us, and make sure that we are being careful with
 % all fields of definitions.}

\begin{remark}\label{rem:p=2order}
Part \ref{item:order} of Lemma \ref{lem:aut} does not hold for $p=2$.
In \cite{GeerVlugt}, Theorem 4.1 it is shown that $\Aut^0(C_R)$ always
contains automorphisms of order $4$ for $h\geq 1$ and $p=2$. See also
\cite{LehrMatignon}, Sect. 7.2 for a concrete example. In Remark
\ref{rem:isoclass} we give a few more details on the differences
between the cases $p=2$ and $p$ odd.
\end{remark}

%Let $P$ be a Sylow-$p$ subgroup of $\Aut^0(C_R)$; such a group is
%called the wild inertia group of $\Aut(C_R)$.  This subgroup is
%significant in determining the Jacobian of $C_R$.  It is now examined
%more carefully.

%\xl{This is Lemma 4.8. Should it be moved here?}\irene{This is a good
%  point for the lemma. We need : center generated by $\rho$ +
%  commutation rule. Concsequence: $P$ is the extension of an
%  el. ab. $p$-group by the cyclic group of order $p$ generated by
%  $\rho$. Extraspecial of type I just means that there is no element
%  of order $p^2$, which we have already seen.}

%\irene{ At
  %this point we should include a description of the Sylow $p$-subgroup
  %$P$ as in Section 4.3 of \cite{LehrMatignon}. Prop. 4.14 explains
  %the difference between $p=2$ and $p$ odd.  For $p$ odd all elements
  %of $P=G_\infty(f)$ have order $p$, whereas for $p=2$ there always
  %are elements of order $4$. The point is that all elementary
  %$2$-groups are abelian. The nonabelian groups of order $8$ are $D_8$
  %and $Q_8$=quaternions. These are extraspecial, and have elements of
  %order $4$. I think it would be interesting to work out Remark 4.15
  %here as well, as this may give some better insight on the structure
  %of the group.}

  %\beth{Didn't work out the remark but could come back to this at
%some point.}
The following result is Theorem 13.3 of \cite{GeerVlugt}, and describes
the group $\Aut^0(C_R)$. The structure of the Sylow $p$-subgroup $P$
of $\Aut^0(C_R)$ will be described in more detail in Sect.~\ref{sec:extraspecial} below. Again, we include this result here to provide a proof.

\begin{theorem}[Theorem 13.3 of \cite{GeerVlugt}]\label{thm:aut} \hfill
\begin{enumerate}
\item \label{item:unique} The group $\Aut^0(C_R)$ has a unique Sylow $p$-subgroup, which we
    denote by $P$. It is the subgroup consisting of all automorphisms $\sigma_{1,b,c,1}$ and
    has cardinality $p^{2h+1}$.
\item \label{item:cyclic} The automorphisms $\sigma_{a,0,0,d}$ form a cyclic subgroup $H\subset
    \Aut^0(C_R)$ of order
\begin{equation*}
 \frac{e(p-1)}{2}  \gcd_{\substack{i \geq 0 \\ a_i \neq 0}} (p^i+1),
\end{equation*}
where $e=2$ if all of the indices $i$ such that $a_i \neq 0$ have the
same parity, and $e=1$ otherwise.
\item \label{item:semidirect} The group $\Aut^0(C_R)$ is the semi-direct product of the normal
    subgroup $P$ and the subgroup $H$.
\end{enumerate}
\end{theorem}

\begin{proof} \hfill
\begin{enumerate}
 \item 
To prove part \ref{item:unique}, one easily checks that $\{\sigma_{1,b,c,1} :  \sigma_{1,b,c,1}
\in \Aut^0(C_R)\}$ is a subgroup of $\Aut^0(C_R)$. (This is similar to the proof of Lemma
\ref{lem:commutator} below.) The statements on the order of $\sigma_{a,b,c,d}$ in part
\ref{item:order} of Lemma \ref{lem:aut} imply that $P$ is the unique Sylow $p$-subgroup of
$\Aut^0(C_R)$, which implies that $P$ is a normal subgroup.

Parts \ref{cinW} and \ref{convenientb} of Proposition \ref{prop:uniqueness} imply that the
cardinality of $P$ is equal to $|W|p$. The last statement of part \ref{item:unique} therefore
follows from part \ref{item:prop:b-part3} of Proposition \ref{prop:B}, since $E$ is a separable
polynomial of degree $p^{2h}$.

%\xl{The proof of part (b) wasn't quite right as written (old proof is commented out).}
\item 
To prove part \ref{item:cyclic}, we consider all elements $(a,0,d) \in S(XR(X))$. Part
\ref{degreeB} of Proposition \ref{prop:uniqueness} implies that the polynomial $B_0$ corresponding
to this tuple is zero. Part \ref{item:prep:uniqueness-part1} of Proposition \ref{prop:uniqueness}
therefore implies that $(a,0,d)\in S(XR(X))$ if and only if $aR(aX)=dR(X)$. This condition is
equivalent to $d = a^{p^i+1}$ for all $0 \leq i \leq h$ with $a_i \neq 0$, as can be readily seen
by comparing coefficients in $aR(aX)$ and $dR(X)$. Part \ref{item:cyclic} now follows
immediately.

%To prove (b) we consider all elements $(a,0,d)\in S(XR(X))$. Statements (\ref{degreeB}) and (d)
%of Proposition \ref{prop:uniqueness} imply that the polynomial $B$ and the constant $b$ corresponding
%to this tuple are both zero.  Proposition \ref{prop:uniqueness}.(1) therefore implies that $(a,0,d)\in S(XR(X))$
%if and only if $aR(aX)=dR(X)$. This condition is equivalent to $d = a^{p^i+1}$ for all $0 \leq i \leq h$ with $a_i \neq 0$.
%This can be readily seen by comparing coefficients in $aR(aX)$ and $dR(X)$. Statement (b) follows immediately from this.
\item
Note that the order of $H$ is prime to $p$. In particular, we have $H\cap P=\{1\}$. Part
\ref{item:semidirect}  follows since $\Aut^0(C_R)$ is generated by $H$ and $P$.
\end{enumerate}
\hfill $\qed$\end{proof}

%Stichtenoth's result on Weierstrass semi-groups \cite[Satz 6]{Stichtenoth73}
%(also see \cite{LehrMatigon} Theorem 3.1) shows that in fact, except for two
%cases, $\Aut^0(C_R)=\Aut(C_R)$.

For completeness we state the following theorem, which follows from
\cite{Stichtenoth73}, Satz 6 and Satz 7. (See also Theorem 3.1 of
\cite{LehrMatignon}.) Since we study the automorphism group of $C_R$
over the algebraically closed field $k$ here, it is no restriction to
assume that $R(X)$ is monic.

\begin{theorem}\label{thm:semigroups} Let $R$ be monic.
\begin{enumerate}
\item Assume that $R(X)\notin \{ X, X^{p}\}.$ Then $\Aut(C_R)=\Aut^0(C_R)$.
%\xl{Should this be $R(X)\notin \{ X, X^{p}\}.$? I always assumed that (b) and (c) referred to the two cases that were not covered by (a)}
\item  If $R(X)=X^p$, then  $\Aut(C_R)=\PGU_3(p)$
\item If  $R(X)=X$, then  $\Aut(C_R)\simeq \SL_2(p)$.
\end{enumerate}
 \end{theorem}
%\xl{I don't understand the phrase "as described in Theorem \ref{thm:aut}." Do
%we mean, "and $\Aut^0(C_R)$ is as described in Theorem \ref{thm:aut}?"}

%\begin{remark}\label{rem:auth=0}\irene{Iam not sure whether this remark is
%really helpful. It is refered to in the proof of Lemma \ref{lem:Fqmodels}.}
%For future reference we describe the automorphism group of the curve
%$C_R$ for $R(X)=X$. In this case $C_R$ is hyperelliptic: the
%hyperelliptic involution is $\iota\colon C_R\to C_R,\, (x,y)\mapsto
%(-x, y).$ As in the proof of Proposition 3.3 of \cite{LehrMatignon}
%one sees that every automorphism of $C_R$ induces an automorphism of
%$C_R/\langle \iota\rangle\simeq \PP^1_Y,$ where $\PP^1_Y$ is the
%projective line with coordinate $Y$.  The isomorphism $\Aut(C_R)\simeq
%\SL_2(p)$ from Theorem \ref{thm:semigroups}.(c) identifies $\iota$
%with $-I$. Moreover, $\SL_2(p)$ acts on $\PP^1_Y$ as M\"obius
%transformations.  In particular, it follows that $\Aut(C_R)$ acts
%transitively on the set $C_R(\FF_p)$ of $\FF_p$-rational points of
%$C_R$.
%\end{remark}

%\irene{To do: Remark that all elements of $P$ may be defined over the
%  splitting field of $E$. The following remark is essentially the
%  proof of this. }

%\irene{The following has been moved here from Section
%  \ref{sec:quotientcurve}, since it was refered to in Section
%  \ref{sec:KR}. I am not sure whether this is the best place.}
%  \beth{I think it makes sense in this section, with some preamble.}

For future reference we note the following result on the higher
ramification groups of the point $\infty\in C_R$ in the cover $C_R\to
C_R/\Aut^0(C_R)$. For the definition of the higher ramification groups
and their basic properties we refer to \cite{SerreCL}, Chap. 4 or
\cite{Stichtenoth}, Chap. 3.

%\xl{I think that this lemma is about the ramification groups of the point $\infty$ in the cover $C_R\to C_R/\Aut^0(C_R)$,
%\emph{not} $C_R\to C_R/\Aut(C_R)$ as claimed before, since these two groups might not be equal.}
%\irene{There was a typo in the statement of (a) below. The proof was correct, though.}

\begin{lemma}\label{lem:Aquotient} Let $R$ be an additive polynomial of degree
 $h\geq 1$, and $C_R$ as given in (\ref{eq:CR}).
\begin{enumerate}
\item \label{item:ram} The filtration of higher ramification groups in the lower numbering of
    $\Aut^0(C_R)$ is
\[
G=G_0=\Aut^0(C_R)\supsetneq P=G_1\supsetneq G_2=\cdots =
G_{1+p^h}=\langle \rho\rangle\supsetneq \{1\}.
\]
\item \label{item:genus} Let $H\subset \Aut(C_R)$ be any subgroup which contains $\rho$. Then
    $g(C_R/H)=0$.
%\item[(c)] Let $A<P$ be an abelian subgroup of order $p^h$
%with $A\cap Z(P)=\{1\}$.  The curve $\overline{C}_A:=C_R/A$ has genus
%$(p-1)/2$ and may be given by an Artin--Schreier equation
%\begin{equation}\label{eq:Aquotient}
%Y^p-Y=f_A(X),
%\end{equation}
%where $f_A(X)$ is a polynomial of degree $2$.
\end{enumerate}
\end{lemma}

\begin{proof}
To prove part \ref{item:ram}, write $\nu_\infty$ for the valuation at the unique point $\infty$ at
infinity and choose a uniformizing parameter $t$ at $\infty$. One easily computes that
\[
\nu_\infty\left(\frac{\sigma(t)-t}{t}\right)=
\begin{cases}
1+p^h&\text{ if }\sigma\in \langle \rho \rangle\setminus\{1\},\\
1&\text{ if }\sigma\in P\setminus \langle \rho \rangle.
\end{cases}
\]
This may also be deduced from the fact that the quotient of $C_R$ by
the subgroup generated by the Artin--Schreier automorphism
$\rho(x,y)=(x, y+1)$ has genus $0$ (\cite[Lemma
2.4]{MatignonRocher}).

Part \ref{item:genus} follows immediately from the fact that the function field of the curve
$C_R/\langle\rho\rangle$ is $k(X)$. This can also be deduced from part \ref{item:ram}.
%
%Let $A$ be as in the statement of the lemma.  Since we assume that
%$A\cap Z(P)=\{1\}$, the Riemann--Hurwitz genus formula yields
%\[
%2g(C_R)-2=p^h(p-1)-2=(2g(\overline{C}_A)-2)p^h+2(p^h-1).
%\]
%We conclude that $g(\overline{C}_A)=(p-1)/2$.
%
%We have shown in Section \ref{sec:auto} that the elements of $A$
%commute with $\rho$. It follows that $\overline{C}_A$ is still an
%Artin--Schreier cover of the projective line branched at one
%point. Statement (c) therefore follows from Artin--Schreier theory.
\hfill $\qed$\end{proof}

\section{Extraspecial groups and the structure of $P$}\label{sec:extraspecial}
%\irene{I decided to adopt the spelling ``extraspecial'', since this is
%  the most common one in the literature}

We now focus again on the subgroup $P$ described in part \ref{item:unique} of Theorem
\ref{thm:aut}. %Recall that $p$ is an odd prime.
Part \ref{item:order} of Lemma \ref{lem:aut} implies that the Sylow $p$-subgroup $P$ of
$\Aut^0(C_R)$ consists precisely of the automorphisms $\sigma_{1,b,c,1}(x,y) = (x+c, y+b+B_c(x))$.
For brevity, we simplify their notation to
\[ \sigma_{b,c} = \sigma_{1,b,c,1}. \]
%
%\renate{We use this notation so much that I thought it should not just be hidden in the text.}

The main result of the section, Theorem \ref{thm:extraspecial}, states
that $P$ is an extraspecial group. For more details on extraspecial
groups we refer the reader to \cite[Chap.  III.13]{Huppert} and
\cite{Suzuki}. Recall that we assume that $p$ is an odd prime. The
classification of extraspecial $2$-groups is different from that for
odd primes.

\begin{definition}\label{def:extraspecial}
A noncommutative $p$-group $G$ is {\em extraspecial} if its center
$Z(G)$ has order $p$ and the quotient $G/Z(G)$ is elementary abelian.
\end{definition}

We denote by $E(p^3)$ the unique nonabelian group of cardinality $p^3$ and
exponent $p$. It can be given by generators and relations as follows:
\[
E(p^3)=\langle x,y \mid x^p=y^p = [x,y]^p = 1, [x,y]\in Z(E(p^3))\rangle.
\]
This group obviously is an extraspecial group.

The following lemma describes the commutation relation in $P$. The
lemma contains the key steps to prove that $P$ is an extraspecial
group.

%\irene{Actually we did not prove before that
%  $Z(P)=\langle\rho\rangle$. We only showed that $\rho\in Z(P)$. (c)
%  of course also follows from Corollary \ref{cor:extraspecial}, but I
%  found it helpful to include the proof, since it is elementary. }

\begin{lemma}\label{lem:commutator}
Assume that $h\geq 1$.
\begin{enumerate}
\item \label{lem:commutator-part1} We have $[\sigma_{b_1,c_1}, \sigma_{b_2,c_2}] =
    \rho^{-\epsilon(c_1,c_2)}$, where
\begin{equation*}
\epsilon(c_1,c_2) = B_{c_1}(c_2)-B_{c_2}(c_1).
\end{equation*}
\item \label{lem:commutator-part2} We have $Z(P)=[P,P]=\langle \rho\rangle$. The quotient group
    $P/Z(P)$ is isomorphic to the space $W$ defined in equation (\ref{eq:W}), where the
    isomorphism is induced by $\sigma_{b,c}\mapsto c$.
\item \label{lem:commutator-part3} Any two non-commuting elements $\sigma,\sigma^\prime$ of $P$
    generate a normal subgroup $E_{\sigma,\sigma^\prime}:=\langle \sigma,\sigma^\prime \rangle$
    of $P$ which is isomorphic to $E(p^3)$.
\end{enumerate}
\end{lemma}

%\xl{I was confused by this proof, so I rewrote some parts of it. These are all %the same arguments,
%but just in my words :) The old version is in extraspecial.%tex if my rewrite doesn't make things better.}\

\begin{proof} \hfill
\begin{enumerate}
 \item 
To prove part \ref{lem:commutator-part1}, we compute that
\[
\sigma_{b,c}^{-1}(x,y)=(x-c, y-b-B_c(x-c)).
\]
We therefore have
\begin{align*}
\sigma_{b_1,c_1}& \sigma_{b_2,c_2} \sigma_{b_1,c_1}^{-1}
\sigma_{b_2,c_2}^{-1}(x,y)  = \sigma_{b_1,c_1} \sigma_{b_2,c_2}
\sigma_{b_1,c_1}^{-1}(x-c_2,y-b_2-B_{c_2}(x-c_2)) \\ & =
\sigma_{b_1,c_1} \sigma_{b_2,c_2}
(x-c_2-c_1,y-b_2-B_{c_2}(x-c_2)-b_1-B_{c_1}(x-c_2-c_1))\\ & =
\sigma_{b_1,c_1} (x-c_1,
y-B_{c_2}(x-c_2)-b_1-B_{c_1}(x-c_2-c_1)+B_{c_2}(x-c_2-c_1))\\ & =
\sigma_{b_1,c_1} (x-c_1, y-b_1-B_{c_1}(x-c_2-c_1)-B_{c_2}(c_1))\\ & =
(x, y-B_{c_1}(x-c_2-c_1)-B_{c_2}(c_1)+B_{c_1}(x-c_1))\\ & = (x,
y+B_{c_2}(c_1)-B_{c_1}(c_2)).
\end{align*}

Since $\sigma_{b_1,c_1} \sigma_{b_2,c_2}\sigma_{b_1,c_1}^{-1} \sigma_{b_2,c_2}^{-1}$ certainly
belongs to $\Aut^0(C_R)$, part \ref{item:stab} of Lemma \ref{lem:aut} implies that
$\sigma_{b_1,c_1} \sigma_{b_2,c_2}\sigma_{b_1,c_1}^{-1} \sigma_{b_2,c_2}^{-1}=\sigma_{a,b,c,d}$ for
some $a$, $b$, $c$ and $d$. From our computation above, $a=d=1$, and $c=0$. Since $c=0$, by part \ref{degreeB}
of Proposition \ref{prop:uniqueness}, $B_c(X) = 0$, which implies that $b=
B_{c_2}(c_1)-B_{c_1}(c_2) \in \FF_p$. %Part \ref{lem:commutator-part1} follows.

\item
Part \ref{lem:commutator-part1} shows that $[P,P]\subset \langle \rho\rangle$. Since $P$ is
noncommutative, we have equality. Because $\rho = \sigma_{1,0}$ and $B_0(X) = 0$ by part
\ref{degreeB} of Proposition \ref{prop:uniqueness}, we have that for any $\sigma_{b,c}$,
\begin{equation*}
\sigma_{b,c} \rho \sigma_{b,c}^{-1} \rho^{-1} = \rho^{B_c(0)} = 1,
\end{equation*}
since $B_c(X)$ is an additive polynomial and therefore has no constant term. Thus $\rho$ commutes
with every element of $P$, and $[P,P]= \langle \rho \rangle \subseteq Z(P)$.

To finish the proof of the first statement of part {\ref{lem:commutator-part2}}, we now show that if $c_1
\neq 0$, then for each automorphism $\sigma_{b_1,c_1}$ there exists an automorphism
$\sigma_{b_2,c_2}$ such that $\sigma_{b_1,c_1}$ and $\sigma_{b_2,c_2}$ do not commute. This shows
that in fact $\langle \rho \rangle = Z(P)$.

Let $c_1\in W\setminus \{0\}$. By part \ref{item:prop:b-part2} of Proposition \ref{prop:B} and part
\ref{item:compare} of Proposition \ref{prop:uniqueness}, $(1,c_1,1) \in S(XR(X))$ and by part
\ref{degreeB} of Proposition \ref{prop:uniqueness}, $B_{c_1}(X)$ has degree $p^{h-1}$. Considering
$c_2=:C$ as a variable, the recursive formulas (\ref{eq:B2}) and (\ref{eq:B3}) for the coefficients
$b_i$ of $B_{C}$ show that $\deg_C(b_i)\leq p^{h+i}$. We conclude that the degree of
$\epsilon(c_1,C)$, when considered as polynomial in $C$, is at most $p^{2h-1}$. Since the
cardinality of $W$ is $p^{2h}$, it follows that there exists a $c_2\in W$, and therefore
$\sigma_{b_2,c_2}\in P$, such that $\epsilon(c_1,c_2)\neq 0$. We conclude that
$Z(P)=\langle\rho\rangle$.

Since $\rho\sigma_{b,c}=\sigma_{b+1,c}$, it follows from part \ref{convenientb} of Proposition
\ref{prop:uniqueness} that the map
 \[
P\to W, \qquad \sigma_{b,c}\mapsto c
\]
is a surjective group homomorphism with kernel $\langle \rho\rangle$. %This proves part \ref{lem:commutator-part2}.

\item
Let $\sigma:=\sigma_{b_1, c_1}, \sigma^\prime:=\sigma_{b_2, c_2}\in P$ be two noncommuting
elements, and write $\epsilon=\epsilon(c_1, c_2)$. Part \ref{lem:commutator-part1} implies
that $\sigma\sigma^\prime=\rho^{-\epsilon}\sigma^\prime\sigma$. Since $\sigma, \sigma^\prime$ and
$\rho$ have order $p$ (part \ref{item:order} of Lemma \ref{lem:aut}), it follows that $\sigma$ and $\sigma^\prime$ generate a subgroup $E(\sigma,
\sigma^\prime)$ of order $p^3$ of $P$, which contains $Z(P)=\langle\rho\rangle$. Since the exponent
of this subgroup is $p$, it is isomorphic to $E(p^3)$.

For an arbitrary element $\sigma_{b,c}\in P$, part \ref{lem:commutator-part1} implies that
$\sigma_{b,c}\sigma\sigma_{b,c}^{-1}\in \langle \rho, \sigma\rangle\subset E(\sigma,
\sigma^\prime)$, and similarly for $\sigma^\prime$ replacing $\sigma$. Thus $E(\sigma,
\sigma^\prime)$ is a normal subgroup, proving part~\ref{lem:commutator-part3}.
\end{enumerate}
\hfill $\qed$\end{proof}

\begin{theorem}\label{thm:extraspecial}
Assume that $h\geq 1$. Then the group $P$ is an extraspecial group of exponent $p$.
\end{theorem}

\begin{proof}
Since $h\geq 1$, part \ref{lem:commutator-part2} of Lemma \ref{lem:commutator} shows that $P$ is
an extraspecial group.
% {\padma{Maybe I am being dense here, but is this is really immediate? We want to say something about how $h \geq 1$ forces $\epsilon(c_1,c_2)$ for some pair of elements}}.  Part \ref{lem:commutator-part2} of Lemma \ref{lem:commutator} implies that $P$ is
%extraspecial,
 Part \ref{item:order} of Lemma \ref{lem:aut} yields that $P$ has
exponent $p$.
\hfill $\qed$\end{proof}

We now show that $P$ is a central product of $h$ copies of $E(p^3)$, i.e., $P$ is isomorphic to the
quotient of the direct product of $h$ copies of $E(p^3)$, where the centers of each copy have been
identified. These subgroups of $P$ of order $p^3$ have been described in part
\ref{lem:commutator-part3} of Lemma \ref{lem:commutator}.

\begin{corollary}\label{cor:extraspecial}
Assume that $h\geq 1$. Then $P$ is a central product of $h$ copies of $E(p^3)$.
\end{corollary}

\begin{proof}
Theorem III.13.7.(c) of \cite{Huppert} states that $P$ is the central product of $h$ extraspecial
groups $P_i$ of order $p^3$. Since $P$ has exponent $p$, it follows that the groups $P_i$ have
exponent $p$ as well. Therefore $P_i\simeq E(p^3)$.
\hfill $\qed$\end{proof}

%\irene{I have deleted the classification result that was here before,
%  since I think it is more useful to explain the structure of $P$ than
%  to recall results that are more general than what we need.}

%\xl{I agree with this.}

%Corollary \ref{cor:extraspecial} implies that any two non-commuting
%elements $\sigma,\sigma^{\prime}$ of $P$ generate a subgroup
%$E_{\sigma,\sigma^{\prime}}:=\langle \sigma,\sigma^{\prime} \rangle$
%isomorphic to $E(p^3)$. Moreover it follows that the center of $P$ is
%contained in $E_{\sigma,\sigma^{\prime}}$.  In our situation this can
%also be seen easily from Lemma \ref{lem:commutator}.(a).
%
% and it follows that
%$E_{\sigma,\sigma^{\prime}}$ is a normal subgroup of $P$.

We describe the decomposition of $P$ as a central product from Corollary \ref{cor:extraspecial}
explicitly; this description is in fact the proof given in \cite[Theorem III.13.7.(c)]{Huppert}.
The proof of part \ref{lem:commutator-part2} of Lemma \ref{lem:commutator}  shows  that
$\epsilon(c_1, c_2)$ defines a nondegenerate symplectic pairing
\[
W\times W\to \FF_p,\qquad (c_1, c_2)\mapsto \epsilon(c_1, c_2).
\]
We may choose a basis $( c_1, \ldots,c_h, c_1', \ldots, c_h') $ of $W$ such that
\[
\epsilon(c_i, c_j')=\delta_{i,j},
\]
where $\delta_{i,j}$ is the Kronecker function. In particular, it follows that $\langle c_1,
\ldots, c_h\rangle\subset W$ is a maximal isotropic subspace of the bilinear form $\epsilon$.

%\beth{ I think the W part would be better placed in section 2.  What
%  do you think? } \irene{I don't want to move the statements on $W$ to
%    Section \ref{sec:pointcount}. In that section we consider a
%    completely different symplectic form, so that would be
%    confusing. Besides the result is only used to describe the group
%    theory.}
%\xl{I agree with Irene. Moving this result would force us to move too many things and would confuse matters quite a bit.}

For every $i$, choose elements $\sigma_i,\sigma^{\prime}_i\in P$ which map to $c_i,c^\prime_i$,
respectively, under the quotient map from part \ref{lem:commutator-part2} of Lemma
\ref{lem:commutator}. This corresponds to choosing an element $b_i$ as in part \ref{convenientb} of
Proposition \ref{prop:uniqueness} for each $i$. Part \ref{lem:commutator-part1} of Lemma
\ref{lem:commutator} implies that $\sigma_i$ does not commute with $\sigma^\prime_i$, but commutes
with $\sigma_j, \sigma^\prime_j$ for every $j\neq i$. Therefore $E_i=\langle \sigma_i,
\sigma^\prime_i\rangle$ is isomorphic to $E(p^3)$ (part \ref{lem:commutator-part3} of Lemma
\ref{lem:commutator}). It follows that $P$ is the central product of the subgroups $E_i$.

We finish this section with a description of the maximal abelian subgroups of $P$. This will be
used in Sect. \ref{sec:KR} to obtain a decomposition of the Jacobian of $C_R$.

\begin{proposition}\label{prop:subgroups}
Let $h\geq 1$.
\begin{enumerate}
\item \label{prop:subgroups-part1} Every maximal abelian subgroup ${\mathcal A}$ of $P$ is an
    elementary abelian group of order $p^{h+1}$, and is normal in $P$.
\item\label{prop:subgroups-part2} Let ${\mathcal A}\simeq (\ZZ/p\ZZ)^{h+1}$ be a maximal
    abelian subgroup of $P$.  For any subgroup $A=A_p\simeq (\ZZ/p\ZZ)^h\subset {\mathcal A}$
    with $A_p\cap Z(P)=\{1\}$ there exist subgroups $A_1, \ldots, A_{p-1}$ of ${\mathcal A}$
    such that
\begin{gather*}
{\mathcal A}=Z(P)\cup A_1\cup\cdots \cup A_p,\\ A_i\simeq
(\ZZ/p\ZZ)^h,\qquad A_i\cap Z(P)=\{1\},\qquad A_i\cap A_j=\{1\} \text{
  if }i\neq j.
\end{gather*}
\item \label{prop:subgroups-part3} Any two subgroups $A$ of ${\mathcal A}$ of order $p^h$ which
    trivially intersect the center of $P$ are conjugate inside $P$.
\end{enumerate}
\end{proposition}

\begin{proof} \hfill
\begin{enumerate}
 \item 
The statement that the maximal abelian subgroups $\mathcal{A}$ of $P$ have order $p^{h+1}$ is
Theorem III.13.7.(e) of \cite{Huppert}. 

\item A maximal abelian subgroup ${\mathcal A}$ is the inverse image of a
maximal isotropic subspace of $W$. Since $P$ has exponent $p$, we conclude that ${\mathcal A}\simeq
(\ZZ/p\ZZ)^{h+1}$ is elementary abelian. Part \ref{lem:commutator-part1} of Lemma
\ref{lem:commutator} and the fact that $\mathcal{A}$ is the inverse image of a maximal isotropic subspace of $W$ imply that ${\mathcal A}$ is a normal subgroup of $P$.
% {\padma{Why is the phrase 'as in the proof of
%part \ref{lem:commutator-part3} of Lemma \ref{lem:commutator}' even here? Doesn't it just follow from the two facts stated at the beginning?}}. 
This proves part
\ref{prop:subgroups-part1}.

Let ${\mathcal A}\subset P$ be a maximal abelian subgroup. Without loss of generality, we may assume that
${\mathcal A}$ corresponds to the maximal isotropic subspace generated by the basis elements $c_1,
\ldots, c_h$ of $W$ as described above. In this case we have ${\mathcal A}=\langle \rho, \sigma_1,
\ldots, \sigma_h\rangle$ where $\sigma_i$ maps to $c_i$ under the map from part
\ref{lem:commutator-part2} of Lemma \ref{lem:commutator}.
%\xl{Do we mean that ${\mathcal A}=\langle \rho, \sigma_1, \ldots, \sigma_h\rang%le$ instead? Otherwise
%$\mathcal{A}$ would not contain $\rho$? Also, otherwise %$\mathcal{A}$ and $A_p$ are identical.}
Define
\[
A_p:=\langle \sigma_1, \ldots \sigma_h\rangle.
\]
This is a subgroup of ${\mathcal A}$ of order $p^h$ such that $A_p\cap Z(P)=\{1\}$.

We define $\tau=\sigma_{b, c_1'+\cdots+c_h'}$, where $b$ is some solution of the equation
\begin{equation*}
b^p-b=(c_1'+\cdots+c_h')R(c_1'+\cdots+c_h')
\end{equation*}
%\xl{This used to say ``$b$ is some solution of the equation $b^p-b=cR(c)$," but% I think we meant for $c = c_1'+\cdots+c_h'$, right?}
as specified in Remark \ref{rem:b-identity}. Let
\[
A_i=\tau^{i}A_p\tau^{-i}, \qquad i=1, \ldots, p-1.
\]
By part \ref{item:prop:b-part2a} of Proposition \ref{prop:B}, $B_c(X)$ is additive in $c$. This
implies that
\begin{equation*}
B_{c_1'+\cdots+c_h'}(X) = \sum_{i=1}^h B_{c_i'}(X).
\end{equation*}
The choice of the basis $c_i, c_i'$ of $W$, together with part \ref{lem:commutator-part1} of Lemma
\ref{lem:commutator} implies therefore that
%\xl{Should we include the remark that $B_c$ is also additive in $c$ in
%  the statement of Proposition \ref{prop:B}?} \irene{I have included
%  this. The disadvantage is that the statement of Proposition
%  \ref{prop:B} has gotten even more complicated. Can you have a look?}
%  \xl{I didn't see any changes to Proposition \ref{prop:B}... so I went ahead a%nd made the change myself. I don't think it's too bad.}

\[
\tau \sigma_i\tau^{-1}=\rho^{-\epsilon(c_i',
  c_i)} \sigma_i=\rho^{\epsilon(c_i,
  c_i')}\sigma_i=\rho \sigma_i.
\]
It follows that $A_i\cap Z(P)=\{1\}$ and $A_i\cap A_j=\{1\}$ if $i\neq j$. By counting, we see that
each non-identity element of $\mathcal{A}$ is contained in exactly one $A_i$. 
%This proves part \ref{prop:subgroups-part2}.

\item
Let $A, A'$ be two subgroups of ${\mathcal A}$ as in the statement of part
\ref{prop:subgroups-part3}. Without loss of generality, we may assume that $A=A_p=\langle \sigma_1, \ldots
\sigma_h\rangle$, as in the proof of part \ref{prop:subgroups-part2}. Then $A'=\langle
\rho^{j_1}\sigma_1 \ldots, \rho^{j_h}\sigma_{h}\rangle$ for suitable $j_i\in \FF_p$. Define $c=\sum_{i=1}^h j_ic_i\in W$ and
choose $b$ with $b^p-b=B_c(c)/2$. As in the proof of part \ref{prop:subgroups-part2} it follows
that $\tau:=\sigma_{b, c}$ satisfies $\tau A\tau^{-1}=A'$.
\end{enumerate}
\hfill $\qed$\end{proof}

\section{Decomposition of the Jacobian of $C_R$}\label{sec:KR}

%\xl{TODO: add in KaniRosen result}

%\irene{Did I write too much explanation on the result of KaniRosen?}

In this section we decompose the Jacobian of $C_R$ over the splitting
field $\FF_q$ of the polynomial $E$. This decomposition allows us to
reduce the calculation of the zeta function of $C_R$ over $\FF_q$ to
that of a certain quotient curve. This quotient curve is computed in
Sect. \ref{sec:quotientcurve}, and Sect. \ref{sec:zeta} combines
these results to compute the zeta function of $C_R$ over $\FF_q$.

The decomposition result (Proposition \ref{prop:KR}) we prove below is
based on the following general result of Kani--Rosen (\cite[Theorem B]{KaniRosen}).

\begin{theorem}[Kani-Rosen \cite{KaniRosen}]\label{thm:KR}
Let $C$ be a smooth projective curve defined over an algebraically
closed field $k$, and $G$ a (finite) subgroup of $ \Aut_k(C)$ such that
$G = H_1 \cup H_2 \cup \ldots \cup H_t$, where the subgroups $H_i \leq G$
satisfy $H_i \cap H_j = \{1\}$ for $i\neq j$. Then we have the isogeny
relation
\begin{equation*}
\Jac(C)^{t-1} \times \Jac(C/G)^g \sim \Jac(C/H_1)^{h_1} \times \cdots
\times\Jac(C/H_t)^{h_t},
\end{equation*}
where $g = \# G$, $h_i = \# H_i$, and $\Jac^n = \Jac \times \cdots\times
\Jac$ ($n$ times).
\end{theorem}

We apply Theorem \ref{thm:KR} to a maximal abelian subgroup ${\mathcal  A}\subset P$. Recall from
part \ref{prop:subgroups-part1} of Proposition \ref{prop:subgroups} that ${\mathcal A}$ is an
elementary abelian $p$-group of order $p^{h+1}$ which contains the center $Z(P)=\langle
\rho\rangle$ of $P$. Part \ref{convenientb} of Proposition \ref{prop:uniqueness} implies that all
automorphisms in ${\mathcal A}$ are defined over $\FF_q$.

Recall from part part \ref{prop:subgroups-part2} of Proposition \ref{prop:subgroups} the existence
of a decomposition
\begin{equation}\label{eq:A}
{\mathcal A}=A_0\cup A_1\cup \cdots \cup A_p,
\end{equation}
where $A_0=\langle \rho\rangle$ is the center of $P$ and for $i\neq 0$ the $A_i$ are elementary
abelian $p$-groups of order $p^h$.

Each group $A_i$ defines a quotient curve $\overline{C}_{A_i}:=C_R/A_i.$ Since all automorphisms in
$A_i$ are defined over $\FF_q$, it follows that the quotient curve $\overline{C}_{A_i}$ together
with the natural map $\pi_{A_i}\colon C_R\to\overline{C}_{A_i}$ may also be defined over $\FF_q$.
The following lemma implies that all curves $\overline{C}_{A_i}$ are isomorphic over $\FF_q$.

\begin{lemma}\label{lem:KR}
Let ${\mathcal A}$ be a maximal abelian subgroup of $P$, and let $A$ and $A'$ be two subgroups of
${\mathcal A}$ of order $p^h$ which have trivial intersection with the center of $P$.  Then the
curves $C_R/A$ and $C_R/A'$ are isomorphic over $\FF_q$.
\end{lemma}

\begin{proof}
Part \ref{prop:subgroups-part3} of Proposition \ref{prop:subgroups} states that the subgroups $A$
and $A'$  are conjugate inside $P$. Namely, we have $A'=\tau A\tau^{-1}$ for an explicit element
$\tau \in P$. The automorphism $\tau$ of $C_R$ induces an isomorphism
\[
\tau \colon C_R/A\to C_R/A'.
\]
 Since  $\tau$ is
defined over $\FF_q$, this isomorphism is defined over $\FF_q$ as well.
%
%\xl{I put in $\tau$ here rather than $\tau'$ to make things just a bit easier t%o read. If there was a reason for $\tau'$, I will change it back.}
\hfill $\qed$\end{proof}

We write $J_R:=\Jac(C_R)$ for the Jacobian variety of $C_R$. Since $C_R$ is defined over $\FF_q$
and has an $\FF_q$-rational point, the Jacobian variety $J_R$ is also defined over $\FF_q$.  The
map $\pi_{A_i}$ induces  $\FF_q$-rational isogenies
\begin{equation}
\label{eq:JacAi}
\pi_{A_i,\ast} \colon J_R\to \Jac(\overline{C}_{A_i}), \qquad
\pi_{A_i}^\ast \colon \Jac(\overline{C}_{A_i})\to J_R.
\end{equation}
The element
\[
\varepsilon_{A_i}=\frac{1}{p^h}\pi_{A_i}^\ast\circ
\pi_{A_i,\ast}\in \End^0(J_R):=\End(J_R)\otimes \QQ
\]
%\renate{Do we ever define the $\End^0$ notation? We should.}
is an idempotent (\cite[Sect. 2]{KaniRosen}) and satisfies the property that
$\varepsilon_{A_i}(J_R)$ is isogenous to $\Jac(\overline{C}_{A_i})$. Note that $p^h$ is the degree
of the map $\pi_{A_i}$.

In the following result we use these idempotents to decompose $J_R$. The same strategy was also
used in \cite[Sect. 10]{GeerVlugt} in the case that $p=2$. In that source, Van der Geer and Van
der Vlugt give a direct proof in their situation of the result of Kani--Rosen (Theorem
\ref{thm:KR}) that we apply here.

\begin{proposition}\label{prop:KR}
There exists an $\FF_q$-isogeny
\[
J_R\sim_{\FF_q}  \Jac(\overline{C}_{A_p})^{p^{h}}.
\]
\end{proposition}

\begin{proof}
We apply Theorem \ref{thm:KR} to the decomposition
(\ref{eq:A}) of a maximal abelian subgroup ${\mathcal A}$ of $P$.
 This
result shows the existence of a $k$-isogeny
\begin{equation}\label{eq:KR}
J_R^p\times \Jac(C_R/{\mathcal A})^{p^{h+1}}\sim_k \Jac(\overline{C}_{A_0})^p\times
\prod_{i=1}^p \Jac(\overline{C}_{A_i})^{p^{h}}.
\end{equation}
%\xl{Should this not be
%\begin{equation*}
%J_R^p\times \Jac(C_R/{\mathcal A})^{p^{h+1}}\sim \Jac(\overline{C}_{A_0})^p\tim%es
%\prod_{i=1}^p \Jac(\overline{C}_{A_i})^{p^h}.
%\end{equation*}
%since $\# \mathcal{A} = p^{h+1}$ and $\# A_i = p^h$ if $i \neq 0$?}

The groups ${\mathcal A}$ and $A_0$ contain the Artin--Schreier element $\rho$; hence the curves
$C_R/{\mathcal A}$ and $\overline{C}_{A_0}$ have genus zero (part \ref{item:genus} of Lemma
\ref{lem:Aquotient}). 
%\renate{This referred to part 1 (which used to be (a)) before, but I think it
%should be part 2 which is the genus - part.} 
Therefore the Jacobians of these curves are trivial
and may be omitted from (\ref{eq:KR}).

As before, let $\varepsilon_{A_i}\in \End(J_R)$ denote the idempotent
corresponding to $A_i$. Theorem 2 of \cite{KaniRosen} states that the
isogeny relation from (\ref{eq:KR}) is equivalent to the relation
\[
p\,\text{Id}\sim p^{h}(\sum_{i=1}^p \varepsilon_{A_i})\in \End^0(J_R).
\]
Here, as defined on p. 312 of \cite{KaniRosen}, the notation $a\sim b$ means that $\chi(a)=\chi(b)$
for all virtual characters of $\End^0(J_R)$. Since $\End^0(J_R)$ is a $\QQ$-algebra, we may divide
by $p$ on both sides of this relation. Applying Theorem 2 of \cite{KaniRosen} once more yields the
isogeny relation
\begin{equation}\label{eq:isog}
J_R\sim_{k}  \prod_{i=1}^p\Jac(\overline{C}_{A_i})^{p^{h-1}}.
\end{equation}
We have already seen that the isogenies $\pi_{A_i}^\ast$ and $ \pi_{A_i,\ast}$ are defined over
$\FF_q$. It follows that the isogeny (\ref{eq:isog}) is defined over $\FF_q$ as well (see also
Remark 6 in Sect. 3 of \cite{KaniRosen}). Since the curves $\overline{C}_{A_i}$, and hence also
their Jacobians, are isomorphic (Lemma \ref{lem:KR}), the statement of the proposition follows.
\hfill $\qed$\end{proof}

\section{Quotients of $C_R$ by elementary abelian $p$-groups}
\label{sec:quotientcurve}
%\irene{set-up slightly changed.}

%\renate{Several props, lemmas and such ``reintroduced'' the polynomial $R(X)$. %I don't think we
%need this since we have established that notation once and for all. So I took t%his out.}

We consider again a maximal abelian subgroup ${\mathcal A}\simeq (\ZZ/p\ZZ)^{h+1}$ of $P$ and
choose $A\subset {\mathcal A}$ with $A\simeq (\ZZ/p\ZZ)^h$ and $A\cap Z(P)=\{1\}$.  In this section
we compute an $\FF_q$-model of the quotient curve $\overline{C}_A=C_R/A$. Lemma \ref{lem:KR}
implies that the $\FF_q$-isomorphism class of the quotient curve does not depend on the choice of
the subgroup $A$.

%as in Proposition \ref{prop:subgroups}.(b).  Corollary \ref{cor:KR}
%will imply that to compute the zeta function of $C_R$ it suffices to
%consider one of these subgroups $A_p$.  We therefore write $A=A_p$ in this
%section. We describe the $\FF_q$-model  of the quotient curve
%$\overline{C}_{A}=C_R/A$.

Since $A\cap Z(P)=\{1\}$, part \ref{item:ram} of Lemma
\ref{lem:Aquotient} implies that the filtration of higher ramification
groups in the lower numbering of $A$ is
\begin{equation*}
A = G_0 = G_1 \supsetneq G_2 = \{1\},
\end{equation*}
%\padma{Lemma \ref{lem:Aquotient}(a) has more than just three groups in the lower numbering filtration.
%Am I missing something? I remember we worked this out in our discussions - I have to look up my notes to see
%what exactly I have written there.}\irene{I corrected the typo in   Lemma \ref{lem:Aquotient}.(a).} \xl{Padma:
%This is not the same lower numbering filtration. In Lemma \ref{lem:Aquotient} we are talking about the filtration
%of the group $\Aut^0(C_R)$. This is the filtration of $A$, which does not contain $Z(P)$, and therefore does not have any elements in $G_{1+p^h}$.}
so the Riemann--Hurwitz formula yields
\[
2g(C_R)-2=p^h(p-1)-2=(2g(\overline{C}_A)-2)p^h+2(p^h-1).
\]
We conclude that $g(\overline{C}_A)=(p-1)/2$.

Proposition \ref{prop:subgroups} implies that the elements of $A$ commute with $\rho$, since
$\rho\in Z(P)$.
%\xl{More simply, can we say that $\rho \in Z(P)$?}
It follows that $\overline{C}_A$ is an Artin--Schreier cover of the projective line branched at one
point. Artin--Schreier theory implies therefore that $\overline{C}_{A}$ may be given by an
Artin--Schreier equation
\[
Y^p-Y=f_A(X),
\]
where $f_A(X)$ is a polynomial of degree $2$. Theorem
\ref{thm:Aquotient} below implies that this polynomial $f_A(X)$ is in
fact of the form $f_A(X) = a_AX^2$ for an explicit constant
$a_A$. These curves are all isomorphic over the algebraically closed
field $k$, but not over $\FF_q$.  The following lemma describes the
different $\FF_q$-models of the curves $Y^p-Y=eX^2$ for $e \in \FF_q$.
%The assumption $p\neq 3$ implies that $g(D_e)\geq 2$, and hence that the automo%rphism group of $D_e$ is finite.

%\irene{We probably should get rid of (a) of the next lemma, since it
%  is not used.} \xl{done.}

%\irene{The proof of the next lemma has been corrected.}

%\xl{In the case where $p=3$, can we guarantee that if $\varphi \colon  D_{e_1} \to D_{e_2}$ is an $\FF_q$-isomorphism,
%then the point $P$ on $D_{e_1}$ mapping to $\infty \in D_{e_2}$ is defined over $\mathbb{F}_q$? If so, we can also assume
%that $\varphi$ maps $\infty$ to $\infty$ in this case: $D_{e_1}$ is an elliptic curve defined over $\mathbb{F}_q$ with identity
%given by $\infty$. Then the map $\sigma \colon Q \mapsto Q-P$ is defined over $\FF_q$ and sends $P$ to $\infty$.}
%\irene{I have inlcuded this.}

\begin{lemma}\label{lem:Fqmodels}
%\begin{itemize}
%\item[(a)] Let $D$ be a curve given by an Artin--Schreier equation $Y^p-Y=f(X)$, where $f(X)=f_2X^2+f_1X+f_0\in \FF_q[X]$ has degree $2$. Then $D$ is $\FF_q$-isomorphic to
%\begin{equation}\label{eq:De}
%D_e:\qquad Y^p-Y=eX^2
%\end{equation}
%for some $e\in \FF_q^\times$ if and only if
%\[
%\Tr_{\FF_q/\FF_p}\left(\frac{f_1^2}{4f_2}-f_0\right)=0.
%\]
%\item[(b)]
For $e \in \mathbb{F}_q$, define the curve $D_e$ by the affine equation
\begin{equation}\label{eq:De}
 Y^p-Y=eX^2.
\end{equation}
Two curves $D_{e_1}$ and $D_{e_2}$ as in (\ref{eq:De}) are isomorphic over $\FF_q$ if and only
if $e_1/e_2$ is the product of a square in $\FF_q^\ast$ with an element of $\FF_p^\ast$. In
particular, over $\overline{\FF}_q$, any two of these curves are isomorphic.
%\item[(c)] Over $\overline{\FF}_q$ any curve $D$ as in (a) is isomorphic to $D_1$.
%\end{itemize}
\end{lemma}

\begin{proof}
%Let $D$ be as in the statement of the lemma and write $f(X)=f_2X^2+f_1X+f_0$, where $f_i\in \FF_q$ and $f_2\neq 0$.  Any $\FF_q$-isomorphism $\varphi:D\to D_e$ sends the unique point of $D$ at $\infty$ to the corresponding point of $D_e$. Since $f$ has degree $2$, it follows that any such isomorphism can be written as  $\varphi(x, y)=(\nu_0 x+\nu_1, \nu_2 y+\nu_3)$ with $\nu_i\in \FF_q$ and $\nu_2\nu_0\neq 0$. The condition that $\varphi$ maps $D$ to $D_e$ is equivalent to
%\begin{equation}\label{eq:Fqmodels}
%\begin{split}
%\nu_2^p&=\nu_2,\\
%\nu_2f_2&=e\nu_0^2,\\
%\nu_2f_1&=2e\nu_0\nu_1,\\
%\nu_3^p-\nu_3&=e\nu_1^2-\nu_2f_0.
%\end{split}
%\end{equation}

%The first equation is equivalent to $\nu_2\in \FF_p$. To decide whether $D$ is isomorphic to some $D_e$ it is therefore no restriction to assume that $\nu_2=1$. Defining $e$ by the second equation and $\nu_1$ by the third equation of (\ref{eq:Fqmodels}), we find that $D$ is $\FF_q$-isomorphic to $D_e$ if and only if there exists $\nu_3\in \FF_q$ such that
%\[
%\nu_3^p-\nu_3=-\frac{f_1^2}{4f_2}+f_0.
%\]
%Statement (a) follows. (This is just completing the square.)

Let $D_{e_1}$ and $D_{e_2}$ be curves of the form
(\ref{eq:De}). Suppose there exists an $\FF_q$-isomorphism $\varphi
\colon D_{e_1} \to D_{e_2}$. We claim that there exists an
$\FF_q$-isomorphism which sends $\infty\in D_{e_1}$ to $\infty\in
D_{e_2}$.

We first consider the case that $p> 3$, i.e., $g(D_{e_i})\geq
2$. In this case, Proposition 3.3 of \cite{LehrMatignon} states that
there exists an automorphism $\sigma$ of $D_{e_1}$ over
$\overline{\FF}_q$ such that $\varphi\circ \sigma$ sends the point
$\infty\in D_{e_1}$ to the point $\infty\in D_{e_2}$. To prove the
claim it suffices to show that $\sigma$ may be defined over $\FF_q$.

To prove this, we follow the proof of Proposition 3.3 of
\cite{LehrMatignon} and use the fact that $\varphi$ maps every point
of $D_{e_1}$ to a point of $D_{e_2}$ with the same polar
semigroup. Theorem 3.1.(a) of \cite{LehrMatignon} implies that the
only points of $D_{e_1}$ with the same polar semigroup as $\infty$ are
the points $Q_i:=(0,i)$ with $i\in \FF_p$. It follows that
$\varphi^{-1}(\infty)$ is either $\infty$ or $Q_i$ for some $i\in
\FF_p$. In the former case,  there is nothing to show.  If
$\varphi^{-1}(\infty)=Q_i$, we may choose
\[
\sigma(x,y)=\left(\frac{x}{y^{(p+1)/2}},  \frac{iy-1}{y}\right).
\]
 Note that this is an automorphism of $D_{e_1}$ which maps $\infty$ to
 $Q_i$. Moreover, $\sigma$ is defined over the field of definition of
 $D_{e_1}$, and we are done. %This proves the  claim for $p > 3$.

We now prove the claim in the case that $p=3$. In this case the curves $D_{e_i}$ are
elliptic curves. The inverse $\varphi^{-1} \colon D_{e_2}\to D_{e_1}$
of $\varphi$ is also defined over $\FF_q$. It follows that
$Q:=\varphi^{-1}(\infty)\in D_{e_1}(\FF_q)$ is $\FF_q$-rational. Then
the translation $\tau_{Q-\infty}\colon P\mapsto P+Q-\infty$ is defined
over $\FF_q$ and sends the unique point $\infty\in D_{e_1}$ to
$Q$. Precomposing $\varphi$ with $\tau_{Q-\infty}$ gives an
$\FF_q$-isomorphism which sends $\infty\in D_{e_1}$ to $\infty\in
D_{e_2}$.

Therefore, without loss of generality we  let $\varphi\colon
D_{e_1}\to D_{e_2}$ be an $\FF_q$-isomorphism which sends the unique
point of $D_{e_1}$ at $\infty$ to the unique point of $D_{e_2}$ at
$\infty$.
% \xl{Why is that? Is it because this point is the unique branched point of the Artin-Schreier cover?} \irene{hopefully it is clearer now.}
Any such automorphism can be written as $\varphi(x,y)=(\nu_0 x+\nu_1,
\nu_2 y+\nu_3)$ with $\nu_i\in \FF_q$ and $\nu_2\nu_0\neq 0$. The
condition that $\varphi$ maps $D_{e_1}$ to $D_{e_2}$ is equivalent to
\begin{align}\label{eq:Fqmodels}
\nu_2^p &=\nu_2, &\nu_2 e_1&=e_2 \nu_0^2, \\
0&=2e_2 \nu_0\nu_1, &\nu_3^p-\nu_3&=e_2\nu_1^2.
\end{align}
It follows that $\nu_1=0$ and $\nu_2,\nu_3\in \FF_p$. The coefficient $e_2$ is given by
\[
e_2=\frac{\nu_2 e_1}{\nu_0^2}.
\]
This proves the first assertion of the lemma. The second assertion is clear since any element of
$\overline{\FF}_q^\ast$ is a square in $\overline{\FF}_q^\ast$.
%Statement (c) follows immediately from (a), since the polynomial $X^p-X-\alpha\in \overline{\FF}_q[X]$ is reducible for every number $\alpha$.
\hfill $\qed$\end{proof}

We now compute an $\FF_q$-model of the curve $C_R/A$ for $A\subset P$
an elementary abelian subgroup of cardinality $p^h$ with $A\cap Z(P)=\{1\}$.
We prove this by induction on $h$, following Sect. 13 of \cite{GeerVlugt}. The following proposition is the key step in the
inductive argument. It is a corrected version of Proposition 13.5 of
\cite{GeerVlugt}, which extends to odd $p$ Proposition 9.1
of \cite{GeerVlugt} and is presented without proof. %Van der Geer and Van der Vlugt only indicate the
%generalization of their results to characteristic $p\geq 3$.
Indeed, the formula for the coordinate $V$ of the
quotient curve given in Proposition 13.5 of \cite{GeerVlugt} contains
an error that has been corrected here. We recall that $R(X)$ is an
additive polynomial of degree $p^h$ with leading coefficient $a_h \in
\mathbb{F}_{p^r} \subseteq \FF_q$.

\begin{proposition}\label{prop:pquotient}
Assume that  $h\geq 1$, and let
\[
\sigma(x,y):=\sigma_{b,c}(x,y)=(x+c, y+b+B_c(x))
\]
be an automorphism of $C_R$ with $c\neq 0$ and $b=B_c(c)/2$. Then the
quotient curve $C_R/\langle\sigma\rangle$ is isomorphic over $\FF_q$
to the smooth projective curve given by an affine equation
\begin{equation}\label{eq:pquotient}
V^p-V= \tilde{f}(U)=U \tilde{R}(U),
\end{equation}
where $\tilde{R}(U)\in \FF_q[U]$ is an additive polynomial of degree
$p^{h-1}$ with leading coefficient
\[
\tilde{a}=\begin{cases}
\frac{a_h}{c^{p-1}}&\text{ if } h\neq 1,\\
\frac{a_h}{2c^{p-1}}& \text{ if }h=1.
\end{cases}
\]
\end{proposition}

%\xl{I replaced $a$ with $a_h$ since this is what we use throughout the paper and in the next Theorem (Theorem \ref{thm:Aquotient})}

%\xl{In the case of $h =1$, I get that the leading coefficient is
%\begin{equation*}
%\frac{a_h}{c^{p-1}} - \frac{B(c)^p}{2c^{2p}}= \frac{a_h}{c^{p-1}} - \frac{b_0^p%}{2c^p}  = \frac{a_h}{c^{p-1}} - \frac{a_h}{2c^{p-1}} = \frac{a_h}{2c^{p-1}}
%\end{equation*}
%}
%\irene{I changed this throughout.}

\begin{proof}
In the proof $c$ is fixed, therefore we write $B(X)$ for $B_c(X)$.
%\renate{From what? The only subscript of $c$
%is in $\sigma$, which also has a subscript $b$, and we've already defined $\sigma := \sigma_{b,c}$
%above. I'd take this sentence out.}
 We define new coordinates
\begin{equation}\label{eq:uvdef}
U=X^p-c^{p-1}X, \qquad V=-Y+\Psi(X)=-Y+\gamma X^2+\frac{X}{c}B(X),
\end{equation}
where $\gamma$ is  defined by
\[
\gamma=-\frac{B(c)}{2c^2}. %,\qquad  \delta=\frac{-B(c)+2b}{2c}.
\]
One easily checks that $U$ and $V$ are invariant under $\sigma$. The invariance of $V$ under
$\sigma$ is equivalent to the property
\[
\Psi(X+c)-\Psi(X)=B(X)+b.
\]
Here we use the definition of $b$ as $b=B(c)/2$. Since $U$ and $V$
generate a degree-$p$ subfield of the function field of $C_R$ and the
automorphism $\sigma$ has order $p$, $U$ and $V$ generate the function
field of the quotient curve $C_R/\langle\sigma\rangle$.

From the definition of $U$ and $V$ above, one can see that the Artin--Schreier automorphism $\rho$ induces an automorphism
$\tilde{\rho}(U, V)=(U, V-1)$ on the quotient curve $C_R/\langle
\sigma\rangle$. It follows that the quotient curve is also given by an
Artin--Schreier equation, which we may write as
\begin{equation}\label{eq:quotientcurve}
V^p - V = -Y^p+Y+\Psi^p(X)-\Psi(X) = -XR(X) +\Psi^p(X)-\Psi(X).
\end{equation}
It is clear that the right-hand side of (\ref{eq:quotientcurve}) can be written as a polynomial
$\tilde{f}(U)$ in $U$, since it is invariant under $\sigma$ by construction. Since the constant
term of $\Psi$ is zero, the right-hand side has a zero at $X=0$, so $\tilde{f}(U)\in U\FF_q[U]$.

%Write
%$\tilde{f}(U)=U(\tilde{R}(U)+\tilde{d})$. By comparing degrees of the
%monomials in (\ref{eq:quotientcurve}) we see that $\tilde{R}$ is an
%additive polynomial. (This is similar to the proof of Proposition
%\ref{prop:uniqueness}.(\ref{additive}).)

Recall that part \ref{item:compare} of Proposition
\ref{prop:uniqueness} established
\begin{equation}\label{eq:B}
B(X)^p-B(X)=cR(X)+XR(c).
\end{equation}
This implies
\begin{equation*}
 XR(X)=\frac{X(B(X)^p-B(X))}{c}-\frac{X^2R(c)}{c}.
 \end{equation*}
It follows that
\begin{equation}\label{eq:pquotient1}
-XR(X)+\Psi^p(X)-\Psi(X)=\frac{B(X)^p}{c^p}U+\gamma^pX^{2p}+X^2\left(\frac{R(c)}{c}-\gamma\right).
\end{equation}
%\xl{I don't think there should be a $dX$ here either.}
Using (\ref{eq:B}) one computes
\[
\gamma^pX^{2p}
+X^2\left(\frac{R(c)}{c}-\gamma\right)=
\gamma^pU^2-\frac{B(c)^p}{c^{p+1}}XU.
\]
Define
\[
\Theta(X)=\frac{B(X)^p}{c^p}-\frac{B(c)^p}{c^{p+1}}X.
\]
Since $\Theta$ is invariant under $\sigma$, we may write $\Theta(X)=\theta(U)$ as a polynomial in
$U$. Note that $\theta(0)=0$ since $\Theta(0)=0$. The additivity of the polynomials $B$ and $U$ in
the variable $X$ imply that the polynomial $\theta$ is additive in the variable $U$. It follows
that we may  write $\theta(U)=\sum_{i=0}^{h-1} \mu_i U^{p^i}$.
% \padma{I also remarked that U is
%additive since that is also used to prove additivity of $\theta$ - does this so%und good to you? You
%can just remove this comment if you agree.} \renate{I like it.}
%\padma{I think it would be easier to read if we were to break this into two separate steps. Since $\Theta$
%is invariant under $\sigma$, we have that $\Theta(X) = \theta(U)$ for some polynomial $\theta$. Using the additivity
%of $B$ one can then check that $\theta$ is also additive in the variable $U$ and therefore we may write
%$\theta(U)=\sum_{i=0}^{h-1} \mu_i U^{p^i}$.}
%Using the relationship between the leading respective coefficients $b_{h-1}$ of $B(X)$ and $a_h$ of
%$R(X)$, given by (\ref{eq:B4}),
From (\ref{eq:B4}), we deduce that  the leading coefficient of $\theta$ is
\[
\mu_{h-1}=\frac{b_{h-1}^p}{c^{p}}=\frac{a_h}{c^{p-1}}.
\]
%The $\mu_i$ satisfy \begin{equation}\label{eq:Theta} \begin{split}
%\mu_{h-1}&=-\frac{b_{h-1}^p}{c^{p}},\\ \mu_i&=\mu_{i+1}c^{p^{i+1}(p-1)}-\frac{b_i^p}{c^p},\qquad
%i=0,\ldots, h-2.  \end{split} \end{equation} It follows therefore
%from (\ref{eq:B4}) that the leading term of $\theta(U)$ is
%\begin{equation}\label{eq:ltPsi}
%-\frac{b_{h-1}^p}{c^p}=-\frac{ca_h}{c^p},
%\end{equation}
%where $a_h$ is the leading term of $R$.
%Putting everything together
Altogether, we find
\[
V^p-V=\tilde{f}(U)=U\left(\theta(U)+\gamma^p U\right).
\]
Setting $\tilde{R}(U):=\theta(U)+\gamma^p U $, we see that
$\tilde{R}(U)$ is an additive polynomial in $U$. The statement about
the leading coefficient of $\tilde{R}(U)$ follows from the definitions
of $\theta$ and $\gamma$.
%This finishes the proof.
\hfill $\qed$\end{proof}

\begin{remark}\label{rem:isoclass} We discuss a crucial difference between even and odd characteristic:
Proposition \ref{prop:pquotient} is a statement about the automorphisms
$\sigma_{b,c}$ of order $p$ which are not contained in the center of $P$. For
$p$ odd all elements of $P\setminus Z(P)$ have order $p$. This is not
true for $p=2$, as we already noted in Remark
\ref{rem:p=2order}. Indeed all extraspecial $2$-groups contain
elements of order $4$. The precise structure of the extraspecial group
$P$ in the case that $p=2$ can be found in Theorem 4.1 of
\cite{GeerVlugt}.  The automorphisms $\sigma_{b,c}\in P\setminus Z(P)$
of order $2$ are easily recognized: they satisfy $c\neq 0$ but
$B_c(c)=0$. This observation considerably simplifies the computation
in the proof of Proposition \ref{prop:pquotient}.

The distinction between elements of order $2$ and $4$ in $P\setminus
Z(G)$ in characteristic $2$ yields a decomposition of the
polynomial $E$ (Theorem 3.4 of \cite{GeerVlugt}). There is no
analogous result in odd characteristic.
\end{remark}

Recall from Sect. \ref{sec:extraspecial} that
every maximal abelian subgroup ${\mathcal A}$ of $P$ is the inverse
image of a maximal isotropic subspace $\overline{A}$ of $W$. For any
such $\mathcal{A}$, let $\{c_1,\ldots, c_h\}$ be a basis of
$\overline{A}$ as described prior to Proposition
\ref{prop:subgroups}. Then every subgroup of ${\mathcal A}$ of order
$p^h$ that intersects  $Z(P)$ trivially is generated by
automorphisms of the form $\{ \sigma_{b_1, c_1}, \ldots , \sigma_{b_h,
  c_h} \}$ where $b_i^p - b_i = c_iR(c_i)$ for $1 \leq i \leq h$. In
fact, there is a one-to-one correspondence between such subgroups of
${\mathcal A}$ and sets of elements $\{b_1, \ldots, b_h\}$ satisfying
$b_i^p-b_i=c_iR(c_i)$. By Remark \ref{rem:b-identity} the elements in all these sets are of the form
$b_i=B_{c_i}(c_i)/2 + i$ with $i \in \FF_p$.

% Let $A\subset {\mathcal A}$ be a
%  subgroup with $A\simeq (\ZZ/p\ZZ)^h$ and $A\cap Z(P)=\{1\}$. Recall
%  that we have written $\overline{C}_A=C_R/A$ for the
%  quotient curve.

\begin{theorem}\label{thm:Aquotient}
%\padma{I have reformulated the statement of the theorem itself to clear up the ambiguity in the first
%two paragraphs in the proof. Let me know if you think it is not okay. If it seems okay, then please remove
%this comment and leave the modified version as it is.}
%Let $R(X)$ be an additive polynomial of degree $p^h$ with $h\geq 0$.
Assume $h\geq 0$.  Let ${\mathcal A}$ be a maximal abelian subgroup of
$P$. Any subgroup $A\subset {\mathcal A}$ of order $p^h$ that
intersects the center $Z(P)$ of $P$ trivially gives rise to an
$\mathbb{F}_q$-isomorphism of the quotient curve $\overline{C}_A$ onto the
smooth projective curve given by the affine equation
\begin{equation*}
Y^p-Y = a_{\mathcal A}X^2.
\end{equation*}
Here
\[
a_{\mathcal A}=\frac{a_h}{2}\prod_{c\in
  \overline{A}\setminus \{0\}} c,
\]
for $h\geq 1$, where we recall that $a_h$ is the leading coefficient of $R$ and
$\overline{A}$ is the maximal isotropic subspace of $W$ that is the
image of $\mathcal{A}$ under the quotient map $P\to W$.
For $h=0$, we let 
\[
a_{\mathcal A}=a_0.
\]  
\end{theorem}

%\xl{This is based on Lemma 13.6 of GeerVlugt. In it though, they only
% give the $\overline{\mathbb{F}}_q$ model of $C_R/A_i$. If we knew
% where $P$ lived above, we could strengthen their result.}  work out
% map phi

\begin{proof}
 We prove by induction on $h$ that there exists a subgroup $A\subset
 {\mathcal A}$ with $A\simeq (\ZZ/p\ZZ)^h$ and $Z(P)\cap A=\{1\}$ such
 that the quotient curve $\overline{C}_A=C_R/A$ is given over $\FF_q$
 by the equation stated in the theorem. The statement of the theorem
 follows from this using Lemma \ref{lem:KR}.

For $h=0$ the statement is true by definition.

Assume that $h\geq 1$ and  that the statement of the theorem holds for all additive polynomials
$R(X)$ of degree $p^{h-1}$. Fix a basis $\{ c_1,c_2,\ldots,c_h \}$ for the image of $\mathcal{A}$ in $W$. 
%\padma{I added a line here to say that we are fixing the basis for $\overline{A}$ for the rest of the proof, since the paragraph describing the relation to a basis of $\overline{A}$ has been moved outside the proof.} 
We may choose $b_h=B_{c_h}(c_h)/2$. As in Sect.
\ref{sec:extraspecial}, we write $\sigma_h(x,y)=\sigma_{b_h, c_h}(x,y)=(x+c_h, y+b_h+B_{c_h}(x)).$
Proposition \ref{prop:pquotient} implies that the quotient curve $C_{h-1}:=C_R/\langle
\sigma_h\rangle$ is given by an Artin--Schreier equation
\[
Y_{h-1}^p-Y_{h-1}=X_{h-1}R_{h-1}(X_{h-1}),
\]
where $R_{h-1}$ is an additive polynomial of degree $p^{h-1}$.

Since ${\mathcal A}$ is an abelian group, it follows that ${\mathcal A}_{h-1}:={\mathcal A}/\langle
\sigma_h\rangle\simeq (\ZZ/p\ZZ)^h$ is a maximal abelian subgroup of the Sylow $p$-subgroup
$P_{h-1}$ of $\Aut^0(C_{h-1})$. The definition of the coordinate $X_{h-1}$ as $X^p-c_h^{p-1}X$ in the proof
of Proposition \ref{prop:pquotient} implies that ${\mathcal A}_{h-1}$ corresponds to the maximal
isotropic subspace $\langle \overline{c}_1,\ldots, \overline{c}_{h-1}\rangle$ of
$W_{h-1}:=W/\langle c_h, c_h'\rangle$, where $\overline{c}_i=c_i^p-c_h^{p-1}c_i$ and $c_h'\in W$ is
an element with $\epsilon(c_i, c_h')=\delta_{i,h}$ as in Sect. \ref{sec:extraspecial}.

The induction hypothesis implies that there exists a subgroup $A_{h-1}\subset {\mathcal A}_{h-1}$
with $A_{h-1}\simeq (\ZZ/p\ZZ)^{h-1}$ and $A_{h-1}\cap Z(P_{h-1})=\{1\}$ such that the quotient
$C_{h-1}/A_{h-1}$ is given by
\[
Y_0^p-Y_0=a_{{\mathcal A}_{h-1}} X_0^2.
\]
We may choose $b_i$ satisfying $b_i^p-b_i=c_iR(c_i)$ for $i=1,
\ldots,h-1$ such that the images of $\sigma_{b_1,c_1},\ldots,
\sigma_{b_{h-1},c_{h-1}}$ in ${\mathcal A}_{h-1}$ generate $A_{h-1}$ (Remark \ref{rem:b-identity}). 
%\textcolor{red}{Recall that for each $i$, all the choices for
%  $b_i$ differ by addition of an element in $\FF_p$.} 
Put
$\sigma_i=\sigma_{b_i,c_i}$ for $i=1, \ldots,h-1$. Then $A:=\langle
\sigma_1,\ldots,\sigma_h\rangle$ satisfies
\[
C_R/A\simeq_{\FF_q} C_{h-1}/A_{h-1}.
\]
This concludes the induction proof.
%\textcolor{red}{By Lemma \ref{lem:KR}, the $\FF_q$-isomorphism class of $\overl%ine{C}_A$ only
%depends on ${\mathcal A}$ and not on the particular choice of the subgroup $A_{%h-1}$ of
%$\mathcal{A}$. This concludes the induction proof.}

The statement about $a_{\mathcal A}$ follows immediately from the formula for the leading
coefficient of the quotient curve given in Proposition \ref{prop:pquotient}.
%Since $\overline{b}_{h-1}$ is unique up to addition by an element of
%$\FF_p$, we may assume that it satisfies
%$\overline{b}_{h-1}=\overline{B}_{\overline{c}_i}(\overline{c}_i)/2$
%(Proposition \ref{prop:uniqueness}.(\ref{convenientb})). This amounts
%to replacing $\sigma_{h-1}$ by $\rho^j\sigma_{h-1}$ for some $0\leq j<p$. Note
%that this changes $A$ but not the maximal abelian subgroup ${\mathcal
%  A}$.
%
%\xl{Do we mean instead:\\
%Since $\overline{b}_{h-1}$ is unique up to addition by an element of
%$\FF_p$, we may assume that it satisfies
%$\overline{b}_{h-1}=\overline{B}_{\overline{c}_{h-1}}(\overline{c}_{h-1})/2$?}
%
%By induction we conclude that there exists a subgroup $A$ such that
%the curve $\overline{C}_A=C_R/A$ is given by
%\[
%Y_0^p-Y_0=a_A X_0^2.
%\]
\hfill $\qed$\end{proof}

\section{The zeta function of the curve $C_R$}\label{sec:zeta}

In this section, we describe the zeta function of the curve $C_R$ over the splitting field $\FF_q$
of the polynomial $E(X)$ defined in (\ref{eq:E}).

Let $C$ be a curve defined over a finite field $\FF_{p^s}$, and write $N_n=\# C(\FF_{p^{sn}})$ for
the number of points on $C$ over any extension $\FF_{p^{sn}}$ of $\FF_{p^s}$. Recall that the {\em
zeta function} of $C$, defined as
\[
Z_C(T)=\exp\left(\sum_{n\geq 1}\frac{N_nT^n}{n}\right),
\]
 is
a rational function with the following properties:
\begin{enumerate}
\item The zeta function may be written as
\[
Z_C(T)=\frac{L_{C, \FF_{p^s}}(T)}{(1-T)(1-p^sT)},
\]
where $L_{C, \FF_{p^s}}(T)\in \ZZ[T]$ is a polynomial of degree
$2g(C)$ with constant term $1$.
\item \label{alphas} Write $L_{C, \FF_{p^s}}(T)=\prod_{i=1}^{2g}(1-\alpha_i T)$ with $\alpha_i\in \CC$. After
    suitably ordering  the $\alpha_i$, we have
\[
\alpha_{2g-i}=\frac{p^s}{\alpha_i}, \qquad |\alpha_i|=p^{s/2}.
\]
\item \label{numberofpoints} For each $n$, we have
\begin{equation*}
N_n =  \#C(\FF_{p^{sn}}) = 1 + p^{sn} - \sum_{i =1}^{2g} \alpha_i^n.
\end{equation*}

\item \label{extensions} If
\begin{equation*}
L_{C, \FF_{p^s}}(T) = \prod_{i=1}^{2g}(1-\alpha_i T)
\end{equation*}
 as above, then for any $r \geq 0$, we have
 \begin{equation*}
L_{C, \FF_{p^{rs}}}(T) = \prod_{i=1}^{2g}(1-\alpha_i^r T).
\end{equation*}
\end{enumerate}
The numerator $L_{C, \FF_{p^s}}(T)$ of the zeta function $Z_C(T)$ over $\FF_{p^s}$ is called the
{\em $L$-polynomial} of $C/\FF_{p^s}$.  If the field is clear from the context, we sometimes omit
it from the notation and simply write $L_C(T)$.

Recall that the Hasse--Weil bound asserts that
\[
|\# C(\FF_{p^s})-(p^s+1)|\leq 2p^{s/2}g(C).
\]
A curve $C/\FF_{p^s}$ is called {\em maximal} if $\# C(\FF_{p^s})=p^s+1+2p^{s/2}g(C)$ and {\em
minimal} if $\# C(\FF_{p^s})=p^s+1-2p^{s/2}g(C)$. Since the number of points on a curve must be an
integer, if $C$ is a maximal curve, then $s$ must be even. Furthermore, using properties
\ref{alphas} and \ref{numberofpoints} above, it is clear that $C$ is maximal if $\alpha_j =
-p^{s/2}$ for each $1 \leq j \leq 2g(C)$, and $C$ is minimal if $\alpha_j = p^{s/2}$ for each $1
\leq j \leq 2g(C)$.

Assume that $s$ is even and that $\FF_{p^s}$ is an extension of $\FF_q$. In the notation of
Proposition \ref{prop:quadric}, we have $w_s=\dim_{\FF_p} W=2h$ (Corollary \ref{cor:W}). Since the
curve $C_R$ has genus $p^h(p-1)/2$, Proposition \ref{prop:quadric} implies that $C_R$ is either
maximal or minimal in this case. Moreover, one easily sees that if either $s$ is odd or $\FF_{p^s}$
does not contain $\FF_q$, then $C_R$ is neither maximal nor minimal. The following proposition asserts that this
almost determines the zeta function of $C_R$ over $\FF_q$. The statement is an extension to odd
characteristic of Theorems 10.1 and 10.2 of \cite{GeerVlugt}. Note that the statement for odd
characteristic is simpler than that for characteristic $2$.

\begin{proposition}\label{prop:zeta}
Let $\FF_{p^s}$ be an extension of $\FF_q$, the splitting field of $E(X)$. Write $g=p^h(p-1)/2$ for
the genus of $C_R$.
\begin{enumerate}
\item \label{zeta:part1} If $s$ is even, the $L$-polynomial of $C_R$ is
\[
L_ {C_R}(T)=
(1\pm p^{s/2}T)^{2g}.
\]
\item \label{zeta:part2} If $s$ is odd, the $L$-polynomial of $C_R$ is
\[
L_ {C_R}(T)=
(1\pm p^sT^2)^{g}.
\]
%\begin{align*}
%&(1+p^s T^2)^{g}\qquad \text{ or}\\
%&(1-2p^sT^2+p^{2s}T^4)^{g/2}.
%\end{align*}
\end{enumerate}
\end{proposition}

%\xl{I think that the second case of (b) can be simplified. This looks very satisfying.}
\begin{proof} \hfill
\begin{enumerate}
 \item 
Let $\alpha_1, \ldots, \alpha_{2g}$ be the reciprocal zeros of the $L$-polynomial of $C$ over
$\FF_{p^s}$, where we order the $\alpha_i$ such that $\alpha_i\alpha_{2g-i}=p^s$.

We first assume that $s$ is even. Since $\FF_{p^s}$ is an extension of $\FF_q$, we have
\[
N_1=\#C_R(\FF_{p^s})=1+p^s\pm 2g p^{s/2}=1+p^s-\sum_{i=1}^{2g}\alpha_i.
\]
Since $|\alpha_i|=p^{s/2}$ we conclude that
\[
\alpha_1=\cdots=\alpha_{2g}=\pm p^{s/2}.
\]
This proves part \ref{zeta:part1}.

\item
We now assume that $s$ is odd.  Proposition \ref{prop:quadric} implies that
\begin{equation}\label{eq:sodd}
N_1=\#C_R(\FF_{p^s})=1+p^s=1+p^s-\sum_{i=1}^{2g}\alpha_i.
\end{equation}
Since the reciprocal roots of the $L$-polynomial of $C$ over $\FF_{p^{2s}}$ are $\alpha_j^2$,
we conclude from part \ref{zeta:part1} that either $\alpha_j^2=p^s$ or $\alpha_j^2=-p^s$ for
all $j$.

If $\alpha_j^2=-p^s$ for all $j$, then $\alpha_j=\pm \I p^{s/2}$, where $\I$ is a primitive $4$th
root of unity. It follows that $\alpha_{2g-j}=p^s/\alpha_j=-\alpha_j$. Hence
\[
(1-\alpha_jT)(1-\alpha_{2g-j}T)=1+p^sT^2.
\]
%This corresponds to the first case of (b).

Assume now that $\alpha_j^2=p^s$ for all $j$. In this case we have $\alpha_j=\pm p^{s/2}$ and
$\alpha_{2g-j}=p^s/\alpha_j=\alpha_j$. Let $m= \# \{ 1\leq j\leq g : \alpha_j=p^{s/2}\}$. It
follows from (\ref{eq:sodd}) that
\[
0=\#C_R(\FF_{p^s})-(p^s+1)=p^{s/2}(-2m+2(g-m)).
\]
We conclude that $2g=4m$, i.e., $m=g/2$ (in particular, $g$ is even).  For the $L$-polynomial
of $C_R$ over $\FF_{p^s}$ we find
\[
L_{C_R}(T) = (1-p^sT^2)^g,
\]
%This corresponds to the second case of (b).
%\xl{We have
%\[
%(1-2p^sT^2+p^{2s}T^4)^{g/2} = ((1-p^sT^2)^2)^{g/2} = (1-p^sT^2)^g
%\]}
as claimed in part \ref{zeta:part2}.
\end{enumerate}
\hfill $\qed$\end{proof}

\begin{remark}\label{rem:minmax} \hfill
\begin{enumerate} 
\item The proof of part \ref{zeta:part2} of Proposition \ref{prop:zeta} shows that the
    case $L_ {C_R}(T)=(1- p^sT^2)^{g}$ can only occur when $g$ is even, i.e., if $p\equiv
    1 \ (\bmod{4})$.
\item \label{maxcharacterize} Assume that $s$ is even. Then $\alpha_j
  = p^{s/2}$ or $\alpha_j = - p^{s/2}$ for all $1 \leq j \leq 2g$, and
  therefore $C_R$ is either minimal or maximal. If $C_R$ is minimal
  over $\FF_{p^s}$, each $\alpha_j = p^{s/2}$. The curve $C_R$
  therefore remains minimal over each extension field $\FF_{p^{sf}}$.
  If $C_R$ is maximal over $\FF_{p^s}$, each $\alpha_j =
  -p^{s/2}$. The reciprocal roots of the $L$-polynomial over $\FF_{p^{sf}}$ are
  $\alpha_j^f=(-1)^fp^{sf/2}$. We conclude that $C_R$ is maximal over
  $\FF_{p^{sf}}$ if $f$ is odd and minimal if $f$ is even.
%The  proof of Proposition \ref{prop:zeta} implies that if $C_R$ is minimal, then the reciprocal roots
%of the $L$-polynomial are all $\alpha_j=p^{s/2}$.  Therefore the reciprocal roots of the $L$-polynomial
%over an extension $\FF_{p^{sr}}$ all equal $\alpha_j^r=p^{rs/2}$. We conclude that $C_R$ is minimal over
%any extension of $\FF_{p^s}$.  If $C_R$ is maximal, then the reciprocal roots of the $L$-polynomial are all
%$\alpha_j=-p^{s/2}$. In this case the curve is maximal over the extensions of $\FF_{p^s}$ of odd degree and minimal over the extensions of even degree.
\end{enumerate}
\end{remark}

To determine the zeta function of $C_R$, it remains to decide when the different cases occur.
The following result, which is an immediate corollary of Proposition \ref{prop:KR}, reduces
this problem to the case $h=0$.

\begin{corollary}\label{cor:KR}
Let $A\simeq (\ZZ/p\ZZ)^h\subset P$ be a subgroup with $A\cap
Z(P)=\{0\}$. Write $\overline{C}_{A}=C_R/A$.  Then
\[
L_{C_R, \FF_q}(T)=L_{\overline{C}_{A}, \FF_q}(T)^{p^{h}}.
\]
\end{corollary}

\begin{proof}
This is an immediate consequence of Proposition \ref{prop:KR}, since abelian varieties which
are isogenous over $\FF_q$ have the same zeta function over $\FF_q$. This follows for example
from the cohomological description of the zeta function in Sect. 1 of \cite{Katz}.
\hfill $\qed$\end{proof}

Recall from Theorem \ref{thm:Aquotient} that the curve $\overline{C}_A$ from Corollary
\ref{cor:KR} is a curve of genus $(p-1)/2$ given by an affine equation of the form
\[
Y^p-Y=aX^2,
\]
for some $a\in \FF_q^\ast$. This corresponds to the case $h=0$.  All curves of this form are
isomorphic over $\overline{\FF}_q$, and the different $\FF_q$-models are described in Lemma
\ref{lem:Fqmodels}. The next result determines the $L$-polynomials of the curves
$\overline{C}_A$. In the literature one finds many papers discussing the zeta function of
similar curves using Gauss sums (for example \cite{Coultier}, \cite{Katz},  \cite{Yui}.) We
give a self-contained treatment here based on the results of Sect. \ref{sec:pointcount}.
%\irene{Actually I did not find exactly the
%  statement we need, but it maybe I didn't look carefully enough.}

%\irene{The various cases  should be checked very carefully.}
%\xl{I went through this pretty carefully, and only found one small typo. I expanded on the argument a little bit to make it easier for other people to proofread.}

\begin{theorem}\label{thm:zeta}
%Let $R$ be an additive polynomial.
Consider the curve $C_R$ over some extension of $\FF_q$ and put $g=g(C_R)$. For $h\geq 0$ we
put $a=a_{\mathcal A}$ with $a_{\mathcal A}$ as given in Theorem \ref{thm:Aquotient} for some
choice of $\mathcal{A}$. %For $h=0$ we let $a=a_h$ be the leading coefficient of $R$.
\begin{enumerate}
\item If $p\equiv 1 \ (\bmod{4})$, then the $L$-polynomial of $C_R$ over $\FF_{p^s}$ is given by
\[
L_{C_R, \FF_{p^s}}(T)=
\begin{cases}
(1-p^sT^2)^g & \textup{ if $s$ is odd},\\
(1-p^{s/2}T)^{2g}&
  \textup{ if $s$ is even and $a$ is a square in $\FF_{p^s}^\ast$},\\
  (1+p^{s/2}T)^{2g}& \textup{ if $s$ is even and $a$ is a nonsquare in
    $\FF_{p^s}^\ast$}.
\end{cases}
\]
\item If $p\equiv 3 \ (\bmod{4})$, then the $L$-polynomial of $C_R$ over $\FF_{p^s}$ is given by
\[
L_{C_R, \FF_{p^s}}(T)=
\begin{cases}
(1+p^sT^2)^{g}& \textup{ if $s$ is odd},\\
(1-p^{s/2}T)^{2g} & \textup{ if $s\equiv 0 \ (\bmod{4})$ and $a$ is a square in $\FF_{p^s}^\ast$},\\
(1+p^{s/2}T)^{2g}& \textup{ if $s\equiv 0 \ (\bmod{4})$ and $a$ is a nonsquare in $\FF_{p^s}^\ast$},\\
(1+p^{s/2}T)^{2g} & \textup{ if $s\equiv 2 \ (\bmod{4})$ and $a$ is a square in $\FF_{p^s}^\ast$},\\
(1-p^{s/2}T)^{2g}& \textup{ if $s\equiv 2 \ (\bmod{4})$ and $a$ is a nonsquare in $\FF_{p^s}^\ast$}.
\end{cases}
\]
\end{enumerate}
\end{theorem}

\begin{proof}
Corollary \ref{cor:KR} implies that it suffices to consider the case $h=0$. To prove the theorem we
may therefore assume that $R(X)=aX$.  We label the corresponding curve $D_a$ as we do in Lemma
\ref{lem:Fqmodels}.

\bigskip\noindent
{\bf Case 1:} The element $a$ is a square in $\FF_{p^s}^\ast$.

Then Lemma \ref{lem:Fqmodels} implies that $D_a$ is isomorphic over $\FF_q$ to the curve $D_1$
given by the affine equation $Y^p-Y=X^2$. Since $D_1$ is defined over $\FF_p$, we compute its
$L$-polynomial over $\FF_p$. The argument that we use here proceeds in the same manner as in the
proof of Proposition \ref{prop:quadric}. However, since both the polynomial $R(X)$ and the field
are very simple, we do not need to consider the quadric $Q$ considered in that proof explicitly.

As in the proof of Proposition \ref{prop:zeta}, it suffices to determine the number $N_2$ of
$\FF_{p^2}$-rational points of the curve $D_1$. We have $p+1$ points with $x\in \{0, \infty\}$. As
in the proof of Proposition \ref{prop:quadric}, the $\FF_{p^2}$-points with $x\neq 0, \infty$
correspond to squares $z=x^2$ with $\Tr_{\FF_{p^2}/\FF_p}(z)=0$. Every such element $z$ yields
exactly $2p$ rational points.  Since $\Tr_{\FF_{p^2}/\FF_p}(z)=z+z^p$, the nonzero elements of
trace zero are exactly the elements with $z^{p-1}=-1$. Choosing an element $\zeta\in
\FF_{p^2}^\ast$ of order $2(p-1)$, we conclude that the nonzero elements with trace zero are
\[
\ker(\Tr_{\FF_{p^2}/\FF_p})\setminus\{0\}=\{\zeta^{2j+1} : j=0, \ldots, p-2\}.
\]
First suppose that $p\equiv 3 \ (\bmod{4})$. Then all the elements of $\ker(\Tr_{\FF_{p^2}/\FF_p})$ are
squares in $\FF_{p^2}$, 
so
\[
\#D_1(\FF_{p^2})=1+p+(p-1)2p=1+p^2+(p-1)p.
\]
As in the proof of Proposition \ref{prop:zeta} it follows that $\alpha_j = \pm \I p^{1/2} =
-\alpha_{2g-j}$ for $1 \leq j \leq g$ after suitable relabeling. If $s$ is even then $\alpha_j^s =
\alpha_{2g-j}^s = \I^s p^{s/2}$ and
\begin{equation*}
(1-\alpha_j^sT)(1-\alpha_{2g-j}^sT) = 1 - 2\I ^{s}p^{s/2}T +p^sT^2 =
\begin{cases}
(1-p^{s/2}T)^2 & \textup{if $s \equiv 0 \ (\bmod{4})$},\\
(1+p^{s/2}T)^2 & \textup{if $s \equiv 2 \ (\bmod{4})$}.
\end{cases}
\end{equation*}
If $s$ is odd then $\alpha_j^s = \pm \I^s p^{s/2} = -\alpha_{2g-j}^s$, and therefore
\begin{equation*}
(1-\alpha_j^sT)(1-\alpha_{2g-j}^sT) = 1+p^sT^2.
\end{equation*}
%This concludes the proof in this case.

Now assume that $p\equiv 1 \ (\bmod{4})$. Then none of the elements of $\ker(\Tr_{\FF_{p^2}/\FF_p})$ 
are squares in $\FF_{p^2}$, and we conclude that
\[
\#D_1(\FF_{p^2})=1+p=1+p^2-(p-1)p.
\]
Again as in the proof of Proposition \ref{prop:zeta} it follows that, up to relabeling, $\alpha_j =
p^{s/2} = \alpha_{2g-j}$ for $1 \leq j \leq g/2$, and $\alpha_j = - p^{s/2} = \alpha_{2g-j}$ for
$g/2+1 \leq j \leq g$. (Note that $g$ is even since $p \equiv 1 \ (\bmod{4})$.) We may therefore
relabel again to ensure that $\alpha_j = p^{s/2} = - \alpha_{2g-j}$, for $1 \leq j \leq g$. With
this new labeling, if $s$ is even, then $\alpha_j^s = \alpha_{2g-j}^s = p^{s/2}$, and
\begin{equation*}
(1-\alpha_j^s T)(1-\alpha_{j+g/2}^s T) = (1-p^{s/2}T)^2,
\end{equation*}
and if $s$ is odd then $\alpha_j^s = p^{s/2} = - \alpha_{2g-j}$ and
\begin{equation*}
(1-\alpha_j^s T)(1-\alpha_{j+g/2}^s T) = (1-p^{s/2}T)(1+p^{s/2}T) = (1-p^sT^2).
\end{equation*}
This concludes Case 1.

%$L_{D_1, \FF_p}$ is a power of $(1-pT^2)$. In both cases, the statements of the theorem
%for $\FF_{p^s}$ are easily deduced, by using that the reciprocal roots of the $L$-polynomial over $\FF_{p^s}$ are the $s$th powers of those over $\FF_p$.

\bigskip\noindent
{\bf Case 2:} The element $a$ is a nonsquare in $\FF_{p^s}^\ast$ and  $s$ is odd.

%\frac taken out
Then the set $\{ a\beta^2 : \beta \in \FF_{p^s}^\ast\}$ contains $(p^s-1)/2$ distinct
elements, all of which are nonsquares. As a consequence, this set contains all nonsquares of
$\FF_{p^s}$. For $s$ odd, the nonsquares in $\FF_p^\ast$ are also nonsquares in $\FF_{p^s}^\ast$,
and therefore the set $\{ a\beta^2 : \beta \in \FF_{p^s}^\ast\}$ contains an element in
$\FF_p^\ast$. (In fact, this set contains all the nonsquares in $\FF_p$.) Lemma \ref{lem:Fqmodels}
now implies that the curve $D_a$ is isomorphic over $\FF_q$ to the curve $D_1$, and the desired
result follows therefore from Case~1.

%Since exactly half the elements of $\FF_{p^s}^\ast$ are nonsquare, we may find a $\beta\in \FF_{p^s}^\ast$ such that $\gamma:=a\beta^2\in \FF_p^\ast$.

\bigskip\noindent {\bf Case 3:} The element $a$ is a nonsquare in  $\FF_{p^s}^\ast$ and  $s$ is even.

Here, we consider $M:=\ker(\Tr_{\FF_{p^s}/\FF_p})=\{z\in \FF_{p^s} :
\Tr_{\FF_{p^s}/\FF_p}(z)=0\}.$ Since the trace is surjective and $\FF_p$-linear, the cardinality of
$M$ is $p^{s-1}$. We may write $M$ as a disjoint union
%
%\renate{The $\amalg$ notation is not all that standard. I changed this to simpl%e $\cup$, since we
%said that the union is disjoint.}
\[
%M=\{0\}\amalg M^{\rm\scriptstyle sq}\amalg M^{\rm\scriptstyle nsq},
M=\{0\}\cup M^{\rm\scriptstyle sq}\cup M^{\rm\scriptstyle nsq},
\]
where $M^{\rm\scriptstyle sq}$ (resp.\ $M^{\rm\scriptstyle nsq}$) are
the elements of $M\setminus\{0\}$ which are squares
(resp.\ nonsquares) in $\FF_{p^s}^\ast$.

As in the proof of Case 1 we have
\[
\#D_1(\FF_{p^s})=1+p+2p \, \# M^{\rm\scriptstyle sq},
\]
and a similar argument gives
\[
\#D_a(\FF_{p^s})=1+p+2p \, \# M^{\rm\scriptstyle nsq}.
\]
From the expression for $\#D_1(\FF_{p^s})$ computed in Case 1, it follows that
\[
\# M^{\rm\scriptstyle sq}=
\begin{cases}
\frac{p^{s-1}-1}{2}+\frac{(p-1)}{2}p^{(s-2)/2}& \textup{ if $p\equiv
  3 \ (\bmod{4})$ and
  $s\equiv2 \ (\bmod{4})$},\\ \frac{p^{s-1}-1}{2}-\frac{(p-1)}{2}p^{(s-2)/2}&
\textup{ if $p\equiv 1 \ (\bmod{4})$ or $s\equiv 0 \ (\bmod{4})$}.
\end{cases}
\]
%\xl{ I get
%\[
%\# M^{\rm\scriptstyle sq}=
%\begin{cases}
%\frac{p^{s-1}-1}{2}+p^{(s-2)/2}\frac{(p-1)}{2}& \text{ if $p\equiv 3\Mod{4}$ a%nd $s\equiv2\Mod{4}$},\\
%\frac{p^{s-1}-1}{2}-p^{(s-2)/2}\frac{(p-1)}{2}& \text{ if $p\equiv 1\Mod{4}$ o%r $s\equiv 0\Mod{4}$}.
%\end{cases}
%\]
%This is from
%\begin{equation*}
%\#M^{\rm\scriptstyle sq} = \frac{1}{2p}\left( p^s - p \pm (p-1)p^{s/2}\right)
%\end{equation*}
%}
Since $\# M^{\rm\scriptstyle nsq}=\# M-1-\# M^{\rm\scriptstyle sq}=p^{s-1}-1-\# M^{\rm\scriptstyle
sq}$, we conclude that
\[
\#D_a(\FF_{p^s})=
\begin{cases}
1+p^s-(p-1)p^{s/2}&  \textup{ if $p\equiv 3 \ (\bmod{4})$ and $s\equiv 2 \ (\bmod{4})$},\\
1+p^s+(p-1)p^{s/2}&
\textup{ if $p\equiv 1 \ (\bmod{4})$ or $s\equiv 0 \ (\bmod{4})$}.
\end{cases}
\]
%\xl{ I get the same conclusion here, so it might just be a typo in the expression for $\# M^{\rm\scriptstyle sq}$}
The expressions for the $L$-polynomial now follow as in the previous cases.
\hfill $\qed$\end{proof}

We finish this section by proving that all curves $C_R$ are supersingular. This result is not new.
Our proof just adds some details to Theorem 13.7 in \cite{GeerVlugt}. An alternative proof is given
by Blache (\cite[Corollary 3.7 (ii)]{Blache}).

\begin{proposition}\label{prop:ss}
The curve $C_R$ is supersingular, i.e., its Jacobian is isogenous over $k = \overline{\FF}_q$
to a product of supersingular elliptic curves.
\end{proposition}

\begin{proof}
The curve $C_R$ is supersingular if and only if all the slopes of the Newton polygon of the
$L$-polynomial are $1/2$. (This follows for example from \cite[Theorem 2]{Tate}.)
%{\padma{can you tell me which of the equivalent statements that are in Theorem2(d) of \cite{Tate} implies the stuff about the Newton polygon? Maybe it is not necessary to explain this here, but I would like to understand this better!}} 
The statement of
the proposition  follows therefore from Theorem \ref{thm:zeta} .
\hfill $\qed$\end{proof}

The reasoning of Van der Geer and Van der Vlugt for Theorem 13.7 of \cite{GeerVlugt} is  slightly different,
since they do not compute the $L$-polynomial of $C_R$ over $\FF_q$. They argue that the Jacobian
variety $J_R$ of $C_R$ is isogenous over $k$ to $p^h$ copies of the Jacobian of the curve $D_1$
with equation $Y^p-Y=X^2$. (This is a weaker version of Proposition \ref{prop:KR}.) They then use
the fact that the curve $D_1$ is supersingular.

\section{Examples}\label{sec:examples}

By work of Ihara \cite{Ihara}, Stichtenoth and Xing \cite{StichtenothXing}, and Fuhrmann and Torres
\cite{FuhrmannTorres}, we know that for $q$ a power of a prime, a curve $C$ which is maximal over
$\FF_{q^2}$ satisfies
\begin{equation*}
g(C) \in \left[0,\frac{(q-1)^2}{4}\right] \cup \left\{\frac{q(q-1)}{2}\right\}.
\end{equation*}

%\frac removed
Moreover, the Hermite curves are the only maximal curves of genus $(q(q-1))/2$
\cite{RueckStichtenoth}.

Recall from Sect. \ref{sec:zeta} that a curve $C$ is maximal over $\FF_{p^{2s}}$ if and only if
its $L$-polynomial satisfies $L_{C,\FF_{p^{2s}}} = (1+p^{2s}T)^{2g(C)}$. In our setting, Theorem
\ref{thm:zeta} shows that for a curve $C_R$ of the type considered in this paper and $a$ defined as in
Theorem \ref{thm:zeta}, if $\FF_{p^s}$ contains the splitting field $\FF_q$ of $E(X)$, then $C_R$
is maximal over $\FF_{p^s}$ if and only if one of the following holds:
\begin{itemize}
\item $s$ is even, $a$ is a nonsquare in $\FF_q^\ast$, and $p \equiv 1 \ (\bmod{4})$,
\item $s \equiv 0 \ (\bmod{4})$, $a$ is a nonsquare in $\FF_q^\ast$, and $p \equiv 3 \ (\bmod{4})$, or
\item $s \equiv 2 \ (\bmod{4})$, $a$ is a square in $\FF_q^\ast$, and $p \equiv 3 \ (\bmod{4})$.
\end{itemize}
In each case the negation of the condition on $a$ guarantees that
$C_R$ is a minimal curve over $\FF_{p^s}$.

In light of these facts, the only difficulty in generating examples of
maximal and minimal curves lies in computing suitable elements $a$. In
this section we present certain cases in which such $a$ can be
computed. We start with a discussion of the case $h=0$, and then turn
our attention to $R(X) = X^{p^h}$. For more results along the same
lines we refer to \cite{CO07} and \cite{AnbarMeindl}.
In \cite{CO08} it is shown that all curves $C_R$ that are maximal over
the field $\FF_{p^{2n}}$ are quotients of the Hermite curve $H_{p^n}$
with affine equation $y^{p^n}-y=x^{p^n+1}$. 

At the end of this section we
briefly investigate isomorphisms between certain curves $C_R$ and
curves with defining equations
\begin{equation*}
Y^p + Y = X^{p^h+1}.
\end{equation*}
Throughout this section, we let $H_p$ denote the Hermite curve which is defined by the affine equation
\begin{equation} \label{eq:hermite}
 Y^p+Y=X^{p+1}.
\end{equation}
As mentioned above, this is a maximal curve over $\FF_{p^2}$.
%\xl{I
%  took out the reference for the proof (Lemma 6.4.4 of
%  \cite{Stichtenoth}) since we have already given the original paper
%  source for this fact.}
The curve $Y^p+Y=X^2$ is a quotient of the Hermite curve $H_p$, and therefore this curve is maximal
over $\FF_{p^2}$.
%In the case where
%$h=0$ and the constant $a$ is contained in an extension $\FF_{p^{2s}}$
%of $\FF_p$ of even degree, the curve
%\begin{equation*}
%Y^p-Y = aX^2
%\end{equation*}
%is a twist of the curve $Y^p+Y=X^2$.
The following lemma determines when the twists
\begin{equation*}
Y^p-Y = aX^2
\end{equation*}
of this curve are maximal. A similar result can also be found in Lemma
4.1 of \cite{CO07}. 

\begin{lemma}\label{lem:max}
Let $R(X)=aX\in \FF_{p^{2s}}[X]$. Then $C_R$ is maximal over $\FF_{p^{2s}}$ if and only if one of
the following conditions holds:
\begin{enumerate}
\item $p\equiv 1 \ (\bmod{4})$ and $a\in \FF_{p^{2s}}^\ast$ is a nonsquare,
\item $p\equiv 3 \ (\bmod{4})$, $s$ is even, and $a\in \FF_{p^{2s}}^\ast$ is a nonsquare, or
\item $p\equiv 3 \ (\bmod{4})$, $s$ is odd, and $a\in \FF_{p^{2s}}^\ast$ is a square.
\end{enumerate}
\end{lemma}

\begin{proof}
In this case we have $E(X)= 2aX$, hence $\FF_{p^{2s}}$ automatically contains the splitting field
of $E$. The lemma therefore follows from Theorem \ref{thm:zeta}.
\hfill $\qed$\end{proof}

%\xl{I don't know if we should leave this in the paper, but I thought I
%  should share this remark with you guys :)}\irene{We should
%  definitevely keep this remark. It would be even better to have a
%  larger table indicating more cases where we get results that are not
%  contained in the many-point data base.}

\begin{remark}\label{rem:manypoints}
The database manYPoints (\cite{manypoints}) compiles records of curves
with many points. The following two maximal curves fall in the range
of genus and cardinality covered in the database, and have now been
included in manYPoints. Previously, the database did not state any lower bound for the maximum number of points of a curve of genus $5$ over $\FF_{11^4}$ and a curve of genus $9$ over $\FF_{19^4}$.
\begin{enumerate} 
\item In the case where $h=0$, $p=11$ and $s=4$, let $a \in \FF_{11^4}^\ast$ be a nonsquare. Then the curve
\begin{equation*}
Y^{11} - Y = aX^2
\end{equation*}
is maximal over $\FF_{11^4}$ and of genus $5$. 
\item In the case where $h=0$, $p=19$ and $s = 4$, let $a \in \FF_{19^4}$ be a nonsquare. Then the curve
\begin{equation*}
Y^{19} - Y = aX^2
\end{equation*}
is maximal over $\FF_{19^4}$ and of genus $9$. 
\end{enumerate}
\end{remark}

The following proposition gives an example of a class of maximal
curves with small genus compared to the size of their field of
definition, in contrast to the Hermite curves which have large
genus. A similar result for
$p=2$ can be found in Theorem 7.4 of \cite{GeerVlugt}.  A similar result with $p$ replaced by an arbitrary prime power
can be found in Proposition 4.6 of \cite{CO07}. 

\begin{proposition}\label{prop:exa}
Let $h\geq 1$.
\begin{enumerate}
\item \label{exa-part1} Let $R(X)=X^{p^h}$. Then $E(X)=X^{p^{2h}}+X$, which has splitting field
    $\FF_q=\FF_{p^{4h}}$. The curve $C_R$ is minimal over $\FF_q$.
\item \label{exa-part2} Let $a_h\in \FF_{p^{2h}}^\ast$ be an element with $a_h^{p^h-1}=-1$ and
    define $R(X)=a_hX^{p^h}$. Then $E(X)=a_h^{p^h}(X^{p^{2h}}-X)$, which has splitting field
    $\FF_q=\FF_{p^{2h}}$. The curve $C_R$ is maximal over $\FF_q$.
\end{enumerate}
\end{proposition}

\begin{proof}
We first prove the statement about the splitting field of $E(X)$ for both cases. Consider the
additive polynomial $R(X)=a_hX^{p^h}\in \FF_{p^{s}}[X]$ with $h\geq 1$. Then (\ref{eq:E}) shows
that
\[
E(X)=a_h^{p^h}X^{p^{2h}}+a_hX.
\]
If $a_h=1$, then $E$ has splitting field $\FF_q=\FF_{p^{4h}}$. If $a_h\in \FF_{p^{2h}}^\ast$
satisfies $a_h^{p^h-1}=-1$, then $E(X)=a_h^{p^h}(X^{p^{2h}}-X)$, which has splitting field
$\FF_q=\FF_{p^{2h}}$. %This proves the statements on the splitting fields.
In both cases, we conclude from the explicit expression of $E$ that
\[
W=\{c\in \overline{\FF}_p : c^{p^{2h}}=-a_h^{1-p^h}c\}.
\]
For every $c\in W$, the formulas (\ref{eq:B2}) and (\ref{eq:B3}) imply that
\[
B_c(X)=-\sum_{i=0}^{h-1}a_h^{p^i}c^{p^{h+i}}X^{p^i}.
\]
We first consider the case where $a_h=1$. Choose an element $c\in W\setminus\{0\}$, i.e.,
$c^{p^{2h}}=-c$, and define
% hence
%$B_c(X)=-\sum_{i=0}^{h-1}c^{p^{h+i}X^{p^{i}}.$
%\xl{I still don't think this is true. $B_c(X)$ is a polynomial of degree $p^{h-1}$, and this
%is a polynomial of degree $p^{2h-1}$. I don't see how any relation on $c$ can help. Furthermore,
%the assertions below are still true (with a slight modification) even without this new fact.}
%\irene{This was a stupid typo, we can leave the corrected formula out, of course.}
\[
\overline{A}=\{c\zeta: \zeta\in \FF_{p^h}\}\subset W.
\]
For any two $\zeta_j$, $\zeta_k$ in $\FF_{p^h}$, we have
\[
B_{c\zeta_j}(c\zeta_k)=-\sum_{i=0}^{h-1} \zeta_j^{p^{h+i}}c^{p^{h+i}+p^i} \zeta_k^{p^{i}}
    =-\sum_{i=0}^{h-1} \zeta_j^{p^{i}}c^{p^{h+i}+p^i} \zeta_k^{p^{h+i}}=B_{c\zeta_k}(c\zeta_j),
\]
since $\zeta^{p^h}=\zeta$ for any $\zeta \in \FF_{p^h}$. Therefore the pairing from part
\ref{lem:commutator-part1} of Lemma \ref{lem:commutator} satisfies
\[
\epsilon(c\zeta_j,c\zeta_k)=B_{c\zeta_j}(c\zeta_k)-B_{c\zeta_k}(c\zeta_j)=0 \quad\text{ for any pair
  $(c\zeta_j, c\zeta_k) \in \overline{A}^2$}.
\]
We conclude that $\overline{A}\subset W$ is a maximal isotropic subspace. Write ${\mathcal
A}\subset P$ for the corresponding maximal abelian subgroup of $P$. Recall the constant from
Theorem \ref{thm:Aquotient},
\[
a_{\mathcal A}=\frac{a_h}{2}\prod_{\gamma\in \overline{A}\setminus \{0\}} \gamma,
\]
when $h \geq 1$. Here the leading coefficient $a_h$ of $R(X)$ is $1$. The definition of $\overline{A}$ implies that
\[
\prod_{\gamma\in \overline{A}\setminus
  \{0\}} \gamma=c^{p^h-1}\prod_{\zeta\in \FF_{p^h}^\ast }\zeta=-c^{p^h-1}.
\]
We conclude that $a_{\mathcal A}=-c^{p^h-1}/2$ is a square in $\FF_q^\ast$, since $-1/2$ is a
square in $\FF_{p^2}^\ast\subset \FF_q^\ast$. Theorem \ref{thm:zeta} now yields
\[
L_{C_R,\FF_{q}}(T)=(1-\sqrt{q}T)^{2g}.
\]
It follows that  $C_R$ is minimal over $\FF_q$. %This proves part \ref{exa-part1}.

We now assume that $a_h\in \FF_{p^{2h}}^\ast$ satisfies $a_h^{p^h}=-a_h$. In this case the
splitting field of $E(X)$ is $\FF_q=\FF_{p^{2h}}$ as shown earlier. Choose a primitive
$(p^{2h}-1)$-st root of unity $\zeta$. Then we may write $a_h=\zeta^{(2j+1)(p^h+1)/2}$ for some
$j$. It follows that $a_h\in\FF_q^\ast$ is a square if and only if $(p^h+1)/2$ is even. This is
equivalent to $p\equiv 3 \ (\bmod{4})$ and $h$ odd.

We choose $\overline{A}=\FF_{p^h}\subset W=\FF_{p^{2h}}$. For every $c, c'\in \overline{A}$, we
have
\[
B_{c}(c')= - \sum_{i=0}^{h-1}(a_hcc')^{p^i}=B_{c'}(c).
\]
As in the proof of part \ref{exa-part1}, we conclude that $\overline{A}$ is a maximal isotropic
subspace for the pairing $\epsilon$ from part \ref{lem:commutator-part1} of Lemma
\ref{lem:commutator}. Since
\[
\prod_{c\in \overline{A}\setminus\{0\}} c= -1,
\]
we conclude that $a_{\mathcal A}$ is equivalent to $a_h$ modulo squares in $\FF_q^\ast$.  (The
argument is similar to that in the proof of part \ref{exa-part1}.)  We conclude that $a_{\mathcal
A}$ is a square in $\FF_q^\ast$ if and only of $p\equiv 3 \ (\bmod{4})$ and $h$ is odd. Theorem
\ref{thm:zeta} implies that $C_R$ is a maximal curve over $\FF_q$ in each of these cases. This
proves part \ref{exa-part2}.
\hfill $\qed$\end{proof}

\begin{remark}
 In their follow-up paper \cite{GeerVlugt2} to \cite{GeerVlugt}, Van
 der Geer and Van der Vlugt constructed further examples of maximal
 curves as a fiber product of the curves $C_R$. We have not considered
 this construction in the case of odd characteristic. We leave  this  as a
 subject for future research.
\end{remark}

\begin{example}\label{exa:hermite} \hfill
\begin{enumerate}
\item We consider the Hermite curve $H_p$ given in (\ref{eq:hermite}),
  and the curve $C_R$ given by
\begin{equation*}
Y^p - Y = X^{p+1}.
\end{equation*}
%  The
%curve $H_p$ is a twist of the curve $C_R$ with $R(X)=X^p$. To see this
%we choose elements $\nu_0, \nu_1\in \overline{\FF}_p^\ast$ with
%\[
%\nu_1^p=-\nu_1=\nu_0^{p+1}.
%\]
%One easily sees that $\nu_0\in \FF_{p^4}$ is a element of order
%$2(p^2-1)$ and $\nu_1\in \FF_{p^2}$ has order $2(p-1)$. Then
%\[
%\varphi:H_p\to C_R, \qquad (x,y)\mapsto (\nu_0 x, \nu_1 y)
%\]
%defines an isomorphism over $\FF_{p^4}$.
We claim that the curves $H_p$ and $C_R$ are not isomorphic over $\FF_{p^2}$. To see this, we
show that $\#C_R(\FF_{p^2})=1+p\neq 1+p^3=\#H_p(\FF_{p^2})$. This clearly implies that the two
curves are not isomorphic over $\FF_{p^2}$.

We note that
\[
\psi\colon \FF_{p^2}^\ast\to \FF_{p^2}^\ast,\qquad x\mapsto x^{1+p}
\]
is the restriction of the norm on $\FF_{p^2}/\FF_p$, so the image of
$\psi$ is $\FF_p^\ast$. It follows that
\[
\Tr_{\FF_{p^2}/\FF_p}(x^{1+p})=2x^{1+p}\neq 0 \quad \text{ for all }x\in \FF_{p^2}^\ast.
\]
We conclude that the $\FF_{p^2}$-rational points of $C_R$ are the $p$
points with $x=0$ together with the unique point $\infty$. This proves
the claim. (Exercise 6.7 in \cite{Stichtenoth} asks to prove that $H_p$
and $C_R$ are isomorphic over $\FF_{p^2}$ if $p\equiv 1 \ (\bmod{4})$. The
above calculation shows that this does not hold.)
%\irene{Is this friendly enough?} \xl{This is very friendly :)}
%  \irene{ In Exercise 6.7
%  in \cite{Stichtenoth} it is claimed that the two curves are
%  isomorphic over $\FF_{p^2}$ if $p\equiv 1\Mod{4}$. This seems to be
%  a mistake.}  \xl{I agree, this does seem to be a mistake, except if
%  somehow he meant isomorphic over $\overline{\FF}_p$. (Not sure if
%  this is true in that case, but that seems to be the only hope of
%  making the statement correct.) I imagine that we should make a
%  remark?}

However, the Hermite curve  $H_p$ is isomorphic over $\FF_{p^2}$ to the curve given by
\[
C_{R'}:  Y^p-Y=a_1X^{p+1},
\]
where $a_1\in \FF_{p^2}$ satisfies $a_1^{p-1}=-1$. The isomorphism is given by $\psi\colon
C_{R'}\to H_p, \, (x,y)\mapsto (x, a_1^p y)$. This conforms with part \ref{exa-part2} of
Proposition \ref{prop:exa}.
%One can show that the two
%twists are isomorphic over $\FF_{p^2}$ if and only if $p\equiv
%1\Mod{4}$.  Indeed if $p\equiv 3 \Mod{4}$ one can show that the
%curve $C_R$ is not maximal over $\FF_{p^2}$. Therefore the number of
%$\FF_{p^2}$-rational point of $H_p$ and $C_R$ is different, and the
%curve cannot be isomorphic over $\FF_{p^2}$. The statement that the
%two curve are isomorphic over $\FF_{p^2}$ is Exercise 6.7 in
%\cite{Stichtenoth}.

\item Let $a_h\in\FF_{p^{2h}}^\ast$ be an element with $a_h^{p^h}=-a_h$ as in 
    part \ref{exa-part2} of Proposition \ref{prop:exa}. Write $R(X)=a_hX^{p^h}$. Then
    $\psi\colon (x,y)\mapsto (x, a_h^{p^{2h-1}} y)$ defines an isomorphism between $C_R$ and
    the curve given by
\[
Y^p+Y=X^{p^{h}+1}.
\]
Part \ref{exa-part2} of Proposition \ref{prop:exa} therefore implies that this curve is maximal
over $\FF_{p^{2h}}$. This can also be shown directly, for example using Proposition 6.4.1 of
\cite{Stichtenoth}.
\end{enumerate}
\end{example}

\bibliographystyle{spmpsci.bst}
\bibliography{ASbibliography}

%\vspace{5ex}\noindent {\small
% Irene Bouw, Institute of Pure Mathematics, Ulm University, D-89069 Ulm,  {\tt %irene.bouw@uni-ulm.de}\\
% Wei Ho, Department of Mathematics, University of Michigan, 530 Church Street, %Ann Arbor, MI 48109, {\tt weiho@umich.edu}\\
% Beth Malmskog, Department of Mathematics and Statistics, Villanova University,% 800 Lancaster Avenue, Villanova, PA 19085, \\ {\tt beth.malmskog@villanova.edu%} \\
% Renate Scheidler, Department of Mathematics and Statistics, University of Calg%ary, 2500 University Drive NW, Calgary, Alberta, Canada T2N 1N4, {\tt rscheidl@%ucalgary.ca}. \\
% Padmavathi Srinivasan, Deparment of Mathematics, Massachusetts Institute of Te%chnology, Cambridge, MA 02139-4307, USA. \\ {\tt padma\_sk@math.mit.edu} \\
% Christelle Vincent, Stanford University, Department of Mathematics, 450 Serra %Mall, Building 380, Stanford CA 94305,\\  {\tt cvincent@stanford.edu}\\
%}

\end{document}